\documentclass[letterpaper, 12pt]{amsart}

\usepackage{latexsym,mathrsfs,amscd,amsmath,verbatim,amssymb,calc,enumerate,amsxtra,
	ifpdf,tikz,tikz-cd,pgffor,pgfplots}
\usepackage[all]{xy}
\ifpdf%
  \usepackage{hyperref}%hyperlinks in PDF
\else%
  \usepackage[hypertex]{hyperref}%hyperlinks in DVI; use dvips -z foo.dvi
\fi%

\allowdisplaybreaks[1]

\theoremstyle{plain}
\newtheorem{theorem}{Theorem}[section]
\newtheorem{corollary}[theorem]{Corollary}

\newtheorem{lemma}[theorem]{Lemma}

\newtheorem*{theorem*}{Theorem}

\theoremstyle{remark}
\newtheorem{remark}[theorem]{Remark}

\theoremstyle{definition}

\newtheorem{example}[theorem]{Example}

\numberwithin{equation}{theorem}

{\noindent\ignorespaces}%
{\par\noindent%
\ignorespacesafterend}

\newcommand{\CC}{\mathbb{C}}

\newcommand{\NN}{\mathbb{N}}
\newcommand{\ZZ}{\mathbb{Z}}
\newcommand{\m}{\mathfrak{m}}

\newcommand{\p}{\mathfrak{p}}
\newcommand{\q}{\mathfrak{q}}

\newcommand{\Spec}{\operatorname{Spec}}
\newcommand{\Min}{\operatorname{Min}}
\newcommand{\Max}{\operatorname{Max}}
\newcommand{\Supp}{\operatorname{Supp}}

\newcommand{\Jac}{\operatorname{Jac}}

\newcommand{\Hom}{\operatorname{Hom}}

\newcommand{\GL}{\operatorname{\mathbf{GL}}}
\newcommand{\End}{\operatorname{End}}

\newcommand{\sr}{\operatorname{sr}}

\newcommand{\grade}{\operatorname{grade}}
\newcommand{\opp}{\operatorname{opp}}

\begin{document}

\title[Cancellation of Finite-Dimensional Noetherian Modules]{Cancellation of Finite-Dimensional Noetherian Modules}
\author{Robin Baidya}
\email{rbaidya@utk.edu}
\address{Department of Mathematics, The University of Tennessee, Knoxville, Tennessee 37996}
\curraddr{}
\thanks{}
\date{\today}
%13-XX Commutative rings and algebras
	%13Cxx Theory of modules and ideals
		%13C05 Structure, classification theorems
	%13Exx Chain conditions, finiteness conditions
		%13E05 Noetherian rings and modules
		%13E15 Rings and modules of finite generation or presentation; number of generators
	%13Fxx Arithmetic rings and other special commutative rings
		%13F15 Commutative rings defined by factorization properties (e.g., atomic, factorial, half-factorial)
	%13Hxx Local rings and semilocal rings
		%13H10 Special types (Cohen-Macaulay, Gorenstein, Buchsbaum, etc.)
%14-XX Algebraic geometry
	%14Mxx Special varieties
		%14M05 Varieties defined by ring conditions (factorial, Cohen-Macaulay, seminormal)
%16-XX Associative rings and algebras
	%16Exx Homological methods in associative algebras
		%16E50 von Neumann regular rings and generalizations (associative algebraic aspects)
	%16Sxx Rings and algebras arising under various constructions
		%16S50 Endomorphism rings; matrix rings
\begin{abstract}
The Module Cancellation Problem solicits hypotheses that, when imposed on modules $K$, $L$, and $M$ over a ring $S$, afford the implication $K\oplus L\cong K\oplus M\Longrightarrow L\cong M$.  In a well-known paper on basic element theory from 1973, Eisenbud and Evans lament the ``great scarcity of strong results" in module cancellation research, expressing the wish that, ``under some general hypothesis" on finitely generated modules over a commutative Noetherian ring, cancellation could be demonstrated.  Singling out cancellation theorems by Bass and Dress that feature ``large" projective modules, Eisenbud and Evans contend further that, although ``[s]ome criteria of `largeness' is certainly necessary in general [\ldots,] the need for projectivity is not clear."  In this paper, we prove that cancellation holds if $K$,~$L$, and $M$ are finitely generated modules over a commutative Noetherian ring $S$ such that $K_{\p}^{\oplus (1+\dim (S/\p))}$ is a direct summand of $M_{\p}$ over $S_\p$ for every prime ideal $\p$ of~$S$.  We also weaken projectivity conditions in the cancellation theorems of Bass and Dress and a newer theorem by De~Stefani--Polstra--Yao; in fact, we obtain a statement that unifies all three of these theorems while obviating a projectivity constraint in each one.  To illustrate the scope of our work, we construct a cancellation example that simultaneously eludes the three theorems just mentioned as well as many other observations from the module cancellation literature.
\end{abstract}
\commby{}
\maketitle

\setcounter{section}{+0} %to make sure that the introduction is Section 1

%%%%%%%%%%%%%%%%%%%%%%%%%%%%%%%%%%%%%%%%
\section{Introduction}\label{sec:intro}
%%%%%%%%%%%%%%%%%%%%%%%%%%%%%%%%%%%%%%%%

The Module Cancellation Problem solicits hypotheses that, when imposed on modules $K$,~$L$, and $M$ over a ring $S$, afford the implication $K\oplus L\cong K\oplus M\Longrightarrow L\cong M$.  The problem traces back to a theorem of Frobenius--Stickelberger from 1879 asserting that, up to isomorphism, a finite abelian group is completely determined by a list of its elementary divisors~\cite{FS}.  The Module Cancellation Problem therefore predates the axiomatic definitions of \textit{module} and \textit{ring}, which first appear in a work by Dedekind from 1893~\cite[page~255]{Dor}.

There are essentially four areas of module cancellation research now.

Theorems from one area abide by the premise that $S$ is a module-finite algebra over a commutative ring of dimension at most 2.  From the members of this family, we learn, for instance, that cancellation holds for finitely generated modules over a 0-dimensional commutative ring (Goodearl--Warfield~\cite[Theorem~18]{GW}); finitely generated modules over a Dedekind domain (Hs\"u~\cite[Theorem~1]{Hsu}); and finitely generated torsion-free modules over a 2-dimensional regular affine $\CC$-domain, where $\CC$ denotes the field of all complex numbers (Wiegand~\cite[Theorem~1.2]{Wie1}).  Additional statements of this type abound in the following articles:
\cite{Cha}
\cite{Gur-Lev}
\cite{Hassler0}
\cite{Hassler}
\cite{HW}
\cite{Karr}
\cite{Levy}
\cite{LW}
\cite{Rush}
\cite{Wie1}
\cite{Wie2}.

Accomplishments of another strain require $K$, $L$, or $M$ to be a direct sum of indecomposable modules.  The Krull--Schmidt Theorem~\cite[Theorem~X.7.5]{Lang}, for example, guarantees that cancellation holds for finite-length modules over an associative ring, and a series of observations due to Matlis~\cite[Theorems~2.4 and~2.5 and Propositions~2.7 and~3.1]{Mat} ensures that cancellation holds for injective modules with finite Bass numbers over a commutative Noetherian ring.  A more delicate result of Vasconcelos~\cite[Corollary]{Vasc} attests that cancellation holds if $S$ is a Noetherian normal domain with a torsion-free Picard group, $K$ is a finitely generated $S$-module, and $L:=I^{\oplus n}$ and $M:=J^{\oplus n}$ for some ideals $I$ and $J$ of $S$ and some nonnegative integer $n$.  More information on direct sums of indecomposable modules can be found in the following works:
\cite{Fac}
\cite{Fuchs-Vamos}
\cite{Haefner}
\cite{Levy-Odenthal}.

Advances of a third variety impose finiteness on $\sr(\End_S(K))$, the stable rank of the ring $\End_S(K)$ of all $S$-linear endomorphisms of $K$.  (See Section~\ref{sec:foundations}, Subsection ``Stable rank".)  One such contribution by Evans~\cite[Theorem~2]{Eva} certifies that cancellation holds as long as $\sr(\End_S(K))=1$; another result due to Warfield (as a solo author)~\cite[Theorems~1.2 and~1.6]{War} avouches that cancellation holds if $\sr(\End_S(K))$ is finite and $K^{\oplus\sr(\End_S(K))}$ is a direct summand of $M$ over $S$.  Suslin~\cite[Corollary~8.4]{Sus2} supplies a large collection of examples to which we may apply the cancellation theorems of Evans and Warfield:  \textit{For every nonnegative integer~$d$, every $d$-dimensional affine $\CC$-algebra has stable rank $1+d$.}   The abovementioned theorems of Goodearl--Warfield and Krull--Schmidt embody two additional examples of Evans's Cancellation Theorem.  We direct the reader to the following sources for further information on stable rank:
\cite{Bas}
\cite{Chen}
\cite{CLQ}
\cite{EO}
\cite{Gab}
\cite{Hein}
\cite{Heit2}
\cite{Heit}
\cite{Lam}
\cite{Sta}
\cite{SV} 
\cite{Vas}
\cite{Vas2}.

Findings from a fourth clade of cancellation theorems feature projective modules; here, we focus on three such achievements due to Bass~\cite[Theorem~9.3]{Bas}, Dress~\cite[Theorem~2]{Dress}, and De~Stefani--Polstra--Yao~\cite[Theorem~3.14]{DSPY}.  These three theorems all begin with the supposition that $S$ is an algebra over a commutative ring $R$ with a finite-dimensional Noetherian maximal spectrum $\Max(R)$ or, equivalently (\cite[Corollary following Proposition~1]{Swa}), a finite-dimensional Noetherian $j$-spectrum $j$-$\Spec(R)$. (See Section~\ref{sec:foundations}, Subsection ``Topology".)  With this hypothesis in place, Bass declares that cancellation holds if $S$ is a module-finite $R$-algebra, $K$~is a finitely generated projective $S$-module, and $M$ is an $S$-module admitting a projective direct summand $P$ over $S$ with $S_\m^{\oplus (1+\dim (\Max(R)))}$ a direct summand of $P_\m$ over $S_\m$ for every $\m\in \Max(R)$.  Dress extends Bass by first introducing two new $S$-modules $N$ and $P$:  Dress takes $N$ to be a finitely presented $S$-module such that $\End_S(N)$ is a module-finite $R$-algebra, and Dress assumes that $P$ is a direct summand of $M$ over $S$ such that $N_\m^{\oplus (1+\dim(\Max(R)))}$ is a direct summand of $P_\m$ over $S_\m$ for every $\m\in\Max(R)$.  Dress also supposes that $M$ is a direct summand of a direct sum of finitely presented $S$-modules and that $K$ and $P$ are direct summands of a direct sum of finitely many copies of $N$ over $S$; the last condition on $K$ and $P$ is a projectivity hypothesis since it implies that $\Hom_S(N,K)$ and $\Hom_S(N,P)$ are projective right modules over $\End_S(N)$.  Using all of these constraints, Dress proves the implication $K\oplus L\cong K\oplus M\Longrightarrow L\cong M$.  De~Stefani, Polstra, and Yao extend Bass in a different direction, verifying that cancellation holds if $R$ and $S$ are the same ring, $K$ is a finitely generated projective $S$-module, and $M$ is a finitely generated $S$-module with $S_\p^{\oplus (1+\dim_X(\p))}$ a direct summand of $M_\p$ over $S_\p$ for every $\p\in X:=j$-$\Spec(S)$.  The following articles offer more information on the cancellation properties of various projective modules:
\cite{DK}
\cite{Duc}
\cite{Gupta}
\cite{Heit}
\cite{Kap}
\cite{MS}
\cite{Qui}
\cite{Rao}
\cite{Sta}
\cite{Sus}.

Many cancellation counterexamples complement the preceding results.  One collection of counterexamples due to Kaplansky~\cite[Theorem~3]{Swa2} demonstrates the sharpness of the local hypotheses on~$M$ in Bass, Dress, and De~Stefani--Polstra--Yao:  \textit{For every positive integer $d$ different from $1$,~$3$,~and~$7$, if $S=K$ is the coordinate ring of the real $d$-sphere, $L$ is a free $S$-module of rank~$d$, and $M$ is the $S$-module corresponding to the tangent bundle of the real $d$-sphere, then $K\oplus L\cong K\oplus M$, but $L\not\cong M$.}  The reader can find additional instances of the failure of cancellation in the following works:
\cite{Bas2}
\cite{Cha}
\cite{GLW}
\cite{HW}
\cite{LW}
\cite{Wie1}
\cite{Wie2}.

On one hand, the numerous sources referenced here illustrate that module cancellation theory has grown a significant amount since its origins.  On the other hand, the trends in module cancellation research described above have existed since 1973, suggesting that the terrain of this discipline has not changed dramatically over the last forty years or so.  Accordingly, critical remarks made by Eisenbud and Evans in 1973 on the nature of this field are still relevant today:  In a well-known paper on basic element theory from 1973, Eisenbud and Evans lament the ``great scarcity of strong results" in module cancellation research, expressing the wish that, ``under some general hypothesis" on finitely generated modules over a commutative Noetherian ring, cancellation could be demonstrated~\cite[page~302]{EE}.  Acknowledging that Kaplansky's counterexamples contextualize the ``large" projective modules in Bass and Dress, Eisenbud and Evans nevertheless contend that, although ``[s]ome criteria of `largeness' is certainly necessary in general [\ldots,] the need for projectivity is not clear"~\cite[page~302]{EE}.  By subsequently citing cancellation theorems by Vasconcelos~\cite[Corollary]{Vasc} and Chase~\cite[Theorem~3.7]{Cha} that already avoid projectivity conditions~\cite[page~302]{EE}, Eisenbud and Evans intimate, moreover, that they envision a different theorem still---a cancellation theorem that not only obviates projectivity hypotheses but also accommodates a substantial class of finitely generated modules over every commutative Noetherian ring.

In response to the foregoing entreaty of Eisenbud and Evans, we offer three cancellation results in this paper. 

Theorem~\ref{theorem:short} features a ``general hypothesis" that affords cancellation for finitely generated modules $K$, $L$, and $M$ over a commutative Noetherian ring $S$.   As the reader can verify, our hypothesis weakens the local constraints on $M$ in Bass, Dress, and De~Stefani--Polstra--Yao, consequently boasting sharpness in light of Kaplansky's counterexamples.  The title of our paper honors the fact that, when $S$ is a Jacobson ring, the premises of the following theorem imply that $K$ is a finite-dimensional Noetherian $S$-module.

\begin{theorem}[{cf.~Corollary~\ref{corollary:gen-DSPY}}]\label{theorem:short}
Let $K$, $L$, and $M$ be finitely generated modules over a commutative Noetherian ring $S$ such that $K_{\p}^{\oplus (1+\dim_X(\p))}$ is a direct summand of $M_{\p}$ over $S_\p$ for every prime ideal $\p\in X:=j$-$\Spec(S)\cap\Supp_S(K)$.  Then $K\oplus L\cong K\oplus M\Longrightarrow L\cong M$.   
\end{theorem}

Theorem~\ref{theorem:main}, our main theorem, unifies the cancellation results of Bass, Dress, and De~Stefani--Polstra--Yao while relaxing a projectivity criterion in each.  Since this fact may not be apparent from a simple glance at our main theorem, we prove this point later in the paper by marshalling two corollaries of our main theorem (Corollary~\ref{corollary:gen-Bass} and Corollary~\ref{corollary:gen-DSPY}).  These corollaries reveal, among other things, that we can remove the projectivity constraint on $K$ in Bass and De~Stefani--Polstra--Yao, delete the requirement in Dress that $P$ be a direct summand of a direct sum of finitely many copies of $N$ over~$S$, and replace the projectivity hypothesis on $P$ in Bass with the milder requirement that $P$ be a direct summand of a direct sum of finitely presented $S$-modules.  We do not attempt to weaken Dress's requirement that $K$ be a direct summand of a direct sum of finitely many copies of $N$ over $S$ since we see this hypothesis as an embellishment rather than as a restriction; the chief case of interest is when $K=N$, and our main theorem easily reduces to this case.

\begin{theorem}[Main Theorem, cf.~Theorem~\ref{theorem:main-later}]\label{theorem:main}
	Let $K$, $L$, $M$, and $N$ be right modules over a ring $S$ that is an algebra over a commutative ring $R$, and let $E:=\End_S(N)$ denote the $R$-algebra of all $S$-linear endomorphisms of $N$.  Assume the following:
	\begin{enumerate}
		\item $X:=j$-$\Spec(R)\cap\Supp_R(N)$ is a Noetherian subspace of the Zariski space $\Spec(R)$.
		\item $N$ is a finitely presented $S$-module, and $E$ is a module-finite $R$-algebra.
		\item There is a finitely generated left $E$-submodule $F$ of $\Hom_S(M,N)$ such that, for every $\p\in X$, there is a split surjection in $F_\p^{\oplus (1+\dim_X(\p))}\subseteq\Hom_{S_\p}\left(M_\p,N_\p^{\oplus(1+\dim_X(\p))}\right)$.
	\end{enumerate}
	Then $N\oplus L\cong N\oplus M\Longrightarrow L\cong M$.  More generally, if $K$ is a direct summand of a direct sum of finitely many copies of $N$ over $S$, then $K\oplus L\cong K\oplus M\Longrightarrow L\cong M$.
\end{theorem}

To illustrate the scope of our work, we include a cancellation example that follows from Theorems~\ref{theorem:short} and \ref{theorem:main} but \textit{eludes every other cancellation theorem cited above}.  In preparation for this example, we recall the Jacobian criterion for regularity~\cite[Corollary~16.20]{Eis}, which implies the following fact: \textit{If $S$ is an affine $\CC$-domain of dimension $d\geqslant 1$, then, for every $h\in\{1,\ldots,d\}$, there are infinitely many height-$h$ prime ideals $\q$ of $S$ such that $S_\q$ is factorial.}

\begin{example}[{cf.~Example~\ref{example:1-7}}]\label{example:1}
		\textit{Assume the following:
			\begin{enumerate}
				\item $S$ is an affine $\CC$-domain of dimension $d\geqslant 3$.
				\item $K:=\q$ is a prime ideal of $S$ such that $S_\q$ is factorial and $2\leqslant \textnormal{height}(\q)\leqslant d-1$.
				\item $M:=\q^{\oplus d}\oplus S$.
			\end{enumerate}
			Then, for every $S$-module~$L$, we have $K\oplus L\cong K\oplus M\Longrightarrow L\cong M$.}
\end{example}

\noindent Toward the end of the paper, we explain why this example falls outside the purview of previously published results.

Before then, we establish our main theorem and discuss two of its corollaries.  Section~\ref{sec:foundations} lays the foundation for our work, covering conventions, definitions, and facts employed in subsequent sections.  The proof of our main theorem begins in Section~\ref{sec:sr1} with a study of modules with endomorphism rings of stable rank 1.  In  Section~\ref{sec:unit-regular}, we strengthen our hypotheses slightly, treating the case of a module with an endomorphism ring that is unit-regular modulo its Jacobson radical.  Section~\ref{sec:main-lemma} contains our main lemma (Lemma~\ref{lemma:main}), the statement that constitutes the most difficult step in the proof of our main theorem.  In Section~\ref{sec:main-theorem}, we formulate and certify our main theorem along with two immediate ramifications.  The first of these is Corollary~\ref{corollary:gen-Bass}, which generalizes the cancellation theorems of Bass and Dress; the second is Corollary~\ref{corollary:gen-DSPY}, which generalizes Theorem~\ref{theorem:short} and the De~Stefani--Polstra--Yao Cancellation Theorem.

The last section of our paper (Section~\ref{sec:examples}) addresses why Example~\ref{example:1} evades prior responses to the Module Cancellation Problem.

%%%%%%%%%%%%%%%%%%%%%%%%%%%%%%%%%%%%%%%%%%%%%%%%%%%%%%%%%%%%%%%%%%%%%%%%%%%%%%%%%%%%%%%%%%%%%%
\section{Foundations}\label{sec:foundations}
%%%%%%%%%%%%%%%%%%%%%%%%%%%%%%%%%%%%%%%%%%%%%%%%%%%%%%%%%%%%%%%%%%%%%%%%%%%%%%%%%%%%%%%%%%%%%%

The purpose of this section is to collect conventions, definitions, and facts hailed throughout the rest of the paper.

\subsection*{General conventions}

Every ring is assumed to be associative with unity; every left and right module is taken to be unital; and every module over a commutative ring is presumed to be standard.  The \textit{center} Z$(S)$ of a ring $S$ is the commutative ring of all $a\in S$ such that $ab=ba$ for every $b\in S$.  A ring $S$ is an \textit{algebra} over a commutative ring $R$ with \textit{structure~map} $\sigma:R\rightarrow S$ if $\sigma$ is a ring homomorphism with $\sigma(1_R)=1_S$ and $\sigma(R)\subseteq\textnormal{Z}(S)$.  An algebra~$S$ over a commutative ring $R$ with structure map $\sigma$ is \textit{module-finite over $R$} if $S$ is finitely generated as an $R$-module, relative to $\sigma$.  If $N$ is a right module over a ring $S$, then the ring $\End_S(N)$ of all right $S$-linear endomorphisms of $N$ is understood to act on the left of $N$.  If a ring $S$ is a subset of a ring $E$ with $1_S=1_E$, then $S$ is a \textit{subring of~$E$}.

\subsection*{Associative rings}

Let $E$ be a ring.  The \textit{opposite ring $E^{\opp}$ of $E$} is the underlying abelian group of $E$ equipped with the reversed multiplication operation of $E$.  The ring $E$ is \textit{Noetherian} if every ascending chain of distinct right ideals in $E$ has finite length and every ascending chain of distinct left ideals in $E$ also has finite length; the ring $E$ is \textit{Artinian} if it satisfies the same property as above except with the word \textit{ascending} replaced by the word \textit{descending}.  An element $e$ of $E$ is (an) \textit{idempotent} if $e=e^2$.  An element $u$ of $E$ is a \textit{unit} if there is a (necessarily unique) element $v$ of $E$ with $uv=vu=1$, in which case we call $v$ the \textit{inverse} $u^{-1}$ of $u$.  The ring $E$ is \textit{Dedekind-finite} if, for every element $u$ of $E$, the existence of an element $v$ of~$E$ with $uv=1$ implies that $vu=1$.  A unit of $E$ in Z($E$) is a \textit{central unit}; a subring of Z($E$) is a \textit{central subring of $E$}.  For all positive integers $m$ and $n$, the symbol $\mathbf{M}_{m\times n}(E)$ stands for the set of all $m\times n$ matrices with entries in $E$, and the symbol $\GL_n(E)$ signifies the set of all units in the ring $\mathbf{M}_{n\times n}(E)$.  The \textit{Jacobson radical} $\Jac(E)$ of $E$ is the two-sided ideal of $E$ that is both the intersection of all maximal right ideals of $E$ and the intersection of all maximal left ideals of $E$.   The ring $E$ is \textit{unit-regular} if, for every element $a$ of $E$, there is a unit $u$ of $E$ with $a=aua$.   The ring $E$ is \textit{local} if it has a unique maximal right ideal or, equivalently, if it has a unique maximal left ideal. 

\subsection*{Modules}

Let $M$ and $N$ be right modules over a ring $S$.  The symbol $\Hom_S(M,N)$ refers to the set of all $S$-linear maps from $M$ to $N$.  The set $\Hom_S(M,N)$ is naturally a left module over $E:=\End_S(N)$.  A map $f\in\Hom_S(M,N)$ is a \textit{split surjection with section $g\in\Hom_S(N,M)$} if $fg=1_E$.  If $F$ is a finitely generated left $E$-submodule of $\Hom_S(M,N)$, then $\mu_E(F)$ refers to the minimum number of elements of $F$ needed to generate $F$ over $E$.  For every nonnegative integer $n$, the symbol $N^{\oplus n}$ signifies the direct sum of $n$ copies of $N$ over $S$ with $N^{\oplus 0}:=0$. If $f_1,\ldots,f_n\in\Hom_S(M,N)$ and $g_1,\ldots,g_n\in\Hom_S(N,M)$ for some positive integer $n$, then $(f_1,\ldots,f_n)^{\top}$ stands for the column of maps 
\[
\begin{pmatrix}
f_1\\
\vdots\\
f_n
\end{pmatrix}
\]
naturally representing a member of $\Hom_S(M,N^{\oplus n})$, and $(g_1,\ldots,g_n)$ naturally represents a member of $\Hom_S(N^{\oplus n},M)$.  The module $N$ is \textit{faithful over $S$} if $0_S$ is the only element $a$ of~$S$ such that $xa=0_N$ for every $x\in N$.

\subsection*{Topology}

Let $R$ be a commutative ring.  The \textit{prime spectrum} $\Spec(R)$ of $R$ is the set of all prime ideals of $R$ equipped with the Zariski topology; the \textit{maximal spectrum} $\Max(R)$ of $R$ is the subspace of $\Spec(R)$ consisting of the maximal ideals of $R$; and the \textit{$j$-spectrum} $j$-$\Spec(R)$ of $R$ is the subspace of $\Spec(R)$ composed of the prime ideals of $R$ that are intersections of maximal ideals of~$R$.  A nonempty subset of a topological space $X$ is \textit{irreducible} if it is not the union of two of its proper closed subsets; the \textit{Krull dimension} $\dim(X)$ of $X$ is the supremum of the lengths of chains of distinct closed irreducible sets in the space; for every $\p\in X$, the symbol $\dim_X(\p)$ refers to the dimension of the closure of $\{\p\}$ in $X$; and $X$ is \textit{Noetherian} if every descending chain of distinct closed sets in $X$ has finite length.  By Swan~\cite[Corollary following Proposition~1]{Swa}, $\Max(R)$ and $j$-$\Spec(R)$ have the same dimension, and one of these spaces is Noetherian if and only if the other is.  For every closed set $X$ in $j$-$\Spec(R)$, we let $\Min(X)$ denote the collection of the minimal members of $X$ with respect to set inclusion in~$R$.  If $j$-$\Spec(R)$ is Noetherian, then $\Min(X)$ is finite for every closed set $X$ in $j$-$\Spec(R)$~\cite[page~344]{EO}.  If $N$ is an $R$-module, then the \textit{support of $N$ over $R$}, denoted $\Supp_R(N)$, is the set of all $\p\in\Spec(R)$ with $N_\p\neq 0$.  If $N$ is a right module over an $R$-algebra $S$ such that  $E:=\End_S(N)$ is a module-finite $R$-algebra, then $\Supp_R(N)=\Supp_R(E)$ is closed in $\Spec(R)$, and consequently $j$-$\Spec(R)\cap\Supp_R(N)$ is closed in $j$-$\Spec(R)$.

\subsection*{The $\delta$ operator}

Let $M$ and $N$ be right modules over a ring $S$ that is an algebra over a commutative ring $R$, and let $F$ be a left submodule of $\Hom_S(M,N)$ over $E:=\End_S(N)$.  The symbol $\delta(F)$ signifies the supremum of the nonnegative integers $m$ such that $F^{\oplus m}$, when viewed as a subset of $\Hom_S(M,N^{\oplus m})$, harbors a split surjection.  For every $\p\in\Spec(R)$, the symbol $\delta(F_\p)$ denotes the supremum of the nonnegative integers $m$ for which a split surjection from $M_\p$ onto $N_\p^{\oplus m}$ can be found in $F_\p^{\oplus m}$.  

Suppose that $F:=Ef_1+\cdots+Ef_n$, where $f:=(f_1,\ldots,f_n)^{\top}\in\Hom_S(M,N^{\oplus n})$ for some positive integer $n$, and let $\q\in W\subseteq X:=j$-Spec$(R)\cap\Supp_R(N)$.  We deem $f$ to be \textit{$\q$-split} if $\delta(F_\q)\geqslant\min\{n,1+\dim_X(\q)\}$.  We say that $f$ is \textit{$W$-split} if $f$ is $\p$-split for every $\p\in W$.  The \textit{set $\mathit{\Lambda}$ of test points of $\delta(F_-)$ in $X$} refers to
\[
\bigcup_{m=0}^\infty\Min(Y_m),
\]
where $Y_m:=\{\p\in X:\delta(F_\p)\leqslant m\}$ for every nonnegative integer $m$. 

\subsection*{Commutative rings}

The symbol $\ZZ$ stands for the ring of all rational integers, and $\CC$ signifies the field of all complex numbers.  A \textit{Laurent monomial over a field $k$} is a formal expression $x_1^{g_1}\cdots x_m^{g_m}f$, where $x_1,\ldots,x_m$ are variables; $g_1,\ldots,g_m$ are integers; $m$~is a positive integer; and $f$ is an element of $k$.  A \textit{Laurent polynomial over a field $k$} is a sum of finitely many Laurent monomials over $k$.  An \textit{affine algebra over a field $k$} is a ring of the form $k[x_1,\ldots,x_m]/I$, where $I$ is an ideal of the \textit{polynomial ring $k[x_1,\ldots,x_m]$ in $m$ variables $x_1,\ldots,x_m$ over $k$} and where $m$ is a positive integer.  The \textit{height} of a prime ideal $\p$ in a commutative ring $R$ is the supremum of the lengths of chains of distinct prime ideals in $R$ with maximal member~$\p$.  A \textit{domain} is a commutative ring whose zero ideal is prime.  If $\p$ is a prime ideal of a $d$-dimensional affine domain $S$ over a field, then $\textnormal{height}(\p)+\dim_X(\p)=\dim(X)$, where $X:=j$-$\Spec(S)=\Spec(S)$.  A Noetherian domain is \textit{factorial} if every height-1 prime ideal in the ring is principal.  A domain is \textit{B\'ezout} if every finitely generated ideal of the ring is principal.  A domain is a \textit{principal ideal domain} if every ideal of the ring is principal.  Letting $\NN$ denote the set of all positive integers, we deem a domain $R$ to be \textit{Euclidean} if there is a function $\nu: R\setminus\{0\}\rightarrow\NN$ such that, for all $a,b\in R$ with $b\neq 0$, there are $q,r\in R$ such that $a=bq+r$ and such that either $r=0$ or $\nu(r)<\nu(b)$.

\subsection*{Stable rank}

Let $E$ be a ring.  The \textit{stable rank} $\sr(E)$ of $E$ is the infimum of the positive integers $m$ such that, for all integers $n>m$ and elements $a_1,\ldots,a_n$ of $E$ with \[Ea_1+\cdots+Ea_n=E,\] there are elements $b_1,\ldots,b_{n-1}$ of $E$ with \[E(a_1+b_1a_n)+\cdots+E(a_{n-1}+b_{n-1}a_n)=E.\]
If $S$ is a subring of $E$, then $\sr(S)\leqslant\sr(E)=\sr(E^{\opp})=\sr(E/\Jac(E))$ by Vaserstein~\cite[Lemma~3 and Theorem~2]{Vas} \cite[Theorem~1]{Vas2}.  By Kaplansky~\cite[Lemma~1.7]{Lam}, every ring of stable rank 1 is Dedekind-finite.  By Fuchs--Henriksen--Kaplansky~\cite[Theorem~2.9]{Lam}, every unit-regular ring has stable rank 1.  By Estes--Ohm~\cite[Theorem~5.3]{EO}, every principal ideal domain of stable rank~1 is Euclidean.  For every nonnegative integer $d$, every $d$-dimensional affine $\CC$-algebra has stable rank $1+d$ by Suslin~\cite[Corollary~8.4]{Sus2}.    If $E$ is a module-finite algebra over a commutative ring $R$ with a finite-dimensional Noetherian maximal spectrum, then $\sr(E)\leqslant 1+\dim(\Max(R))$ by Bass's Stable Range Theorem~\cite[Theorem~11.1]{Bas}.

\subsection*{Grade and depth}

Let $M$ be a finitely generated module over a commutative Noetherian ring $S$ with an ideal $I$ such that $IM\neq M$.  An element $a$ of $I$ is a \textit{nonzerodivisor on $M$} if, for every $x\in M$, the equation $ax=0$ implies that $x=0$.  An \textit{$M$-sequence in $I$} is a sequence $a_1,\ldots,a_n$ of elements of $I$ (for some positive integer $n$) such that, for every $m\in\{0,\ldots,n-1\}$, the element $a_{m+1}$ of $I$ is a nonzerodivisor on $M/(a_1M+\cdots+a_mM)$.  The \textit{grade of $I$ on $M$}, denoted $\grade_M(I)$, refers to the maximum of the lengths of $M$-sequences in $I$.  The natural map $M\rightarrow\Hom_S(I,M)$ is an isomorphism if and only if $\grade_M(I)\geqslant 2$.  If $N$ is a finitely generated $S$-module with $IN\neq N$, then $\grade_{M\oplus N}(I)=\min\{\grade_M(I),\grade_N(I)\}$.  If $S$ is a Noetherian factorial domain and $I$ is a prime ideal of height at least 2 in $S$, then $\grade_S(I)\geqslant 2$. If $S$ is a commutative Noetherian local ring with maximal ideal $\m$ such that $\grade_S(\m)\geqslant 1$, then $\grade_\m(\m)=1$.\\

\noindent We begin working toward a proof of our main lemma (Lemma~\ref{lemma:main}) in the next section.

%%%%%%%%%%%%%%%%%%%%%%%%%%%%%%%%%%%%%%%%%%%%%%%%%%%%%%%%%%%%%%%%%%%%%%%%%%%%%%%%%%%%%%%%%%%%%%
\section{The case of a module with an endomorphism ring of stable rank 1}\label{sec:sr1}
%%%%%%%%%%%%%%%%%%%%%%%%%%%%%%%%%%%%%%%%%%%%%%%%%%%%%%%%%%%%%%
%%%%%%%%%%%%%%%%%%%%%%%%%%%%%%%%

Our sole objectives in Sections~\ref{sec:sr1}--\ref{sec:main-lemma} are to prove Lemma~\ref{lemma:DF} and to prove our main lemma (Lemma~\ref{lemma:main}); these observations are the only nonstandard results on which our main theorem (Theorem~\ref{theorem:main-later}) depends.  The proof of Lemma~\ref{lemma:DF} is quick, essentially relying on nothing more than the fact that a ring of stable rank 1 is Dedekind-finite~\cite[Lemma~1.7]{Lam}.  The proof of our main lemma, on the other hand, is quite complicated, resting on a large collection of statements that populate the remainder of this section and the next two sections.  We split this gallery of lemmas into three sections to help the reader keep track of the hypotheses active at various points in our discussion. 

Throughout this section, $M$ and $N$ stand for right modules over a ring $S$, and the ring $E:=\End_S(N)$ is assumed to have stable rank~$1$.   We direct the reader to Section~\ref{sec:foundations} for basic information on stable rank.

\begin{lemma}\label{lemma:DF}
	Let $f\in\Hom_S(M,N)$.  Then $Ef$ contains a split surjection if and only if $f$ is a split surjection.
\end{lemma}

\begin{proof}
	Suppose first that $Ef$ contains a split surjection. Then there is $g\in\Hom_S(N,M)$ with $E(fg)=E$.  Since $E$ is Dedekind-finite, $(fg)E=E$, so there is $h\in E$ with $f(gh)=1$.  Thus $f$ is split surjective, proving the forward implication.  The reverse implication is clear. 
\end{proof}

The proof of our main lemma begins with the next result, which establishes a fundamental connection between $\delta(F)$ and $\mu_E(F)$ for a given finitely generated left $E$-submodule $F$ of $\Hom_S(M,N)$.  The $\delta$ operator is defined in Section~\ref{sec:foundations}.

\begin{lemma}\label{lemma:spl-infty}
	Let $F$ be a finitely generated left $E$-submodule of $\Hom_S(M,N)$.  Then we have $\delta(F)=\infty$ if and only if $\delta(F)>\mu_E(F)$ if and only if $N=0$.
\end{lemma}

\begin{proof}
	If $N=0$, then $\delta(F)=\infty>0=\mu_E(F)$.
	
	If $\delta(F)=\infty$, then $\delta(F)>\mu_E(F)$ since $F$ is finitely generated over $E$ by hypothesis.
	
	Suppose, by way of contradiction, that $\delta(F)>\mu_E(F)$ and $N\neq 0$.  Then $\mu_E(F)>0$.  Let $n:=\mu_E(F)$, and let $f_1,\ldots,f_n\in F$ be such that $F=Ef_1+\cdots+Ef_n$.  Since $\delta(F)>\mu_E(F)$, there is $A\in\mathbf{M}_{(n+1)\times n}(E)$ with $A(f_1,\ldots,f_n)^\top$ split surjective, so $A$ represents an $S$-linear split surjection from $N^{\oplus n}$ onto $N^{\oplus (n+1)}$.  As a result, there is a right $S$-module $Q$ with $N^{\oplus n}\cong N^{\oplus (n+1)}\oplus Q$.  Since $\sr(E)=1$, we can apply Evans's Cancellation Theorem~\cite[Theorem~2]{Eva} iteratively $n$ times to yield $N\oplus Q=0$.  Therefore, $N=0$, a contradiction.
\end{proof}

The next lemma reveals that a version of Gaussian elimination holds for matrices with entries in $E$.  We use this statement to certify Lemma~\ref{lemma:replacement} as well as two results (Lemmas~\ref{lemma:canc-delta-one} and~\ref{lemma:rows}) directly invoked in the proof of our main lemma.   

\begin{lemma}\label{lemma:identity}
Let $(g_1,\ldots,g_n)^\top\in\Hom_S(M,N^{\oplus n})$ for some integer $n\geqslant 2$, and suppose that $\delta(Eg_1+\cdots+Eg_n)\geqslant m$ for some $m\in\{1,\ldots,n\}$.  Then there exists $A\in\mathbf{M}_{m\times n}(E)$ with $A(g_1,\ldots,g_n)^\top\in\Hom_S(M,N^{\oplus m})$ split surjective and with each column of the $m\times m$ identity matrix $I_m$ occupying a predetermined column of $A$ of our choice.  From the special case in which the first $m$ columns of $A$ form $I_m$, we learn that there is a split surjection in the set $g_1+Eg_{m+1}+\cdots+Eg_n$.
\end{lemma}

\begin{proof}
	Assume inductively that, for some $k\in\{0,\ldots,m-1\}$, there is $B:=(b_{i,j})\in\mathbf{M}_{m\times n}(E)$ with $B(g_1,\ldots,g_n)^\top$ split surjective and with $k$ columns of $I_m$ appearing in $B$ at desired locations.  To simplify notation, assume also that $(1,0,\ldots,0)^\top$ is not the leftmost column of $B$ but that we wish for this to be the leftmost column of $A$.  (The general case is similar in spirit but more cumbersome in notation.)  Let $(z_1,\ldots,z_m)\in\Hom_S(N^{\oplus m},M)$ be a section of $B(g_1,\ldots,g_n)^\top$.  Since
	\[
	b_{1,1}(g_1z_1)+\sum_{j=2}^n b_{1,j}(g_jz_1)=1
	\]
	and since $\sr(E)=1$, there are $d,u\in E$ with $u$ a unit such that
	\[
	g_1(z_1u)+\sum_{j=2}^n (db_{1,j})g_j(z_1u)=1.
	\]
	So we get
	\[
	\begin{pmatrix}
	1 & db_{1,2} & \cdots & db_{1,n} \\
	b_{2,1} & b_{2,2} & \cdots & b_{2,n} \\
	\vdots & \vdots & & \vdots \\
	b_{m,1} & b_{m,2} & \cdots & b_{m,n} \\
	\end{pmatrix}
	\begin{pmatrix}
	g_1 \\
	\vdots \\
	g_n \\
	\end{pmatrix}
	\begin{pmatrix}
	z_1u & z_2 & \cdots & z_m \\
	\end{pmatrix}
	=\begin{pmatrix}
	1 	& * 			& \cdots 	& \cdots 	& * \\
	0 			& 1 	& 0 		& \cdots 	& 0 \\
	\vdots 		& 0 			& \ddots 	& \ddots 	& \vdots \\
	\vdots 		& \vdots 		& \ddots 	& \ddots 	& 0 \\
	0 			& 0 			& \cdots	& 0			& 1 \\
	\end{pmatrix}.
	\]
	We can now right-multiply by the inverse of the right side of the equation to produce $I_m$, and we can then conjugate by a matrix in $\GL_m(E)$ to clear $b_{2,1},\ldots,b_{m,1}$ while fixing $I_m$ on the right side of the equation.  This process yields a matrix
	\[
	C
	:=\begin{pmatrix}
	1 & db_{1,2} & \cdots & db_{1,n} \\
	0 & b_{2,2}-b_{2,1}db_{1,2} & \cdots & b_{2,n}-b_{2,1}db_{1,n} \\
	\vdots & \vdots & & \vdots \\
	0 & b_{m,2}-b_{m,1}db_{1,2} & \cdots & b_{m,n}-b_{m,1}db_{1,n} \\
	\end{pmatrix}
	\]
	such that $C(g_1,\ldots,g_n)^\top$ is split surjective.  If $k=0$, then our inductive step is complete; otherwise, let $p\in\{2,\ldots,m\}$ and $q\in\{2,\ldots,n\}$ be such that the $p$th column of $I_m$ is the $q$th column of~$B$.  Then $b_{1,q}=0$, and so inspection reveals that the $q$th column of $C$ is the $q$th column of~$B$.  Therefore, at least $k+1$ columns of $I_m$ appear in $C$ at desired positions.  This completes our inductive step and our proof of the first statement of the lemma.
	
	The second statement of the lemma is an easy corollary of the first.
\end{proof}

The following lemma helps us compare lower bounds for the $\delta$-values of certain left $E$-submodules of $\Hom_S(M,N)$.  We appeal to this result explicitly in the proof of our main lemma.

\begin{lemma}\label{lemma:canc-delta-one}
	Choose $(f_1,\ldots,f_n)^\top\in\Hom_S(M,N^{\oplus n})$ for some integer $n\geqslant 2$, and suppose that $\delta(Ef_1+\cdots+Ef_n)\geqslant m$ for some $m\in\{2,\ldots,n\}$.  Let $(g_1,\ldots,g_n)^{\top}:=U(f_1,\ldots,f_n)^\top$ for some $U\in\mathbf{GL}_n(E)$, and let $k\in\{1,\ldots,m-1\}$.  Then $\delta(Eg_1+\cdots+Eg_{n-k})\geqslant m-k$.
\end{lemma}

\begin{proof}	
	Since $\delta(Ef_1+\cdots+Ef_n)\geqslant m$, there is  $C\in\mathbf{M}_{m\times n}(E)$ with $C(f_1,\ldots,f_n)^\top$ split surjective.	Hence $CU^{-1}(g_1,\ldots,g_n)^\top=C(f_1,\ldots,f_n)^\top$ is split surjective, and so we have $\delta(Eg_1+\cdots+Eg_n)\geqslant m$.  Now, by Lemma~\ref{lemma:identity}, there is
	\[
	A:=
	\setcounter{MaxMatrixCols}{20}
	\begin{pmatrix}
	a_{1,1} & \cdots & a_{1,n-m} & 1 & 0 & \cdots & 0 & 0 &\cdots&\cdots&0\\
	\vdots & & \vdots & 0 & \ddots & \ddots & \vdots & \vdots&&&\vdots \\
	\vdots & & \vdots & \vdots & \ddots & \ddots & 0 & \vdots&&&\vdots \\
	a_{m-k,1} & \cdots & a_{m-k,n-m} & 0 & \cdots & 0 & 1 & 0&\cdots&\cdots&0 \\
	* & \cdots & * & 0 & \cdots & \cdots & 0 &1&0&\cdots&0\\
	\vdots & & \vdots & \vdots & & & \vdots & 0&\ddots&\ddots&\vdots\\
	\vdots & & \vdots & \vdots & & & \vdots & \vdots&\ddots&\ddots&0 \\
	* & \cdots & * & 0 & \cdots & \cdots & 0 & 0&\cdots&0&1 \\
	\end{pmatrix}\in\mathbf{M}_{m\times n}(E)
	\]
	with $A(g_1,\ldots,g_n)^\top$ split surjective.  Define $B$ as the top left $(m-k)\times (n-k)$ submatrix of~$A$:
	\[
	B:=
	\begin{pmatrix}
	a_{1,1} & \cdots & a_{1,n-m} & 1 & 0 & \cdots & 0 \\
	\vdots & & \vdots & 0 & \ddots & \ddots & \vdots \\
	\vdots & & \vdots & \vdots & \ddots & \ddots & 0 \\
	a_{m-k,1} & \cdots & a_{m-k,n-m} & 0 & \cdots & 0 & 1 \\
	\end{pmatrix}\in\mathbf{M}_{(m-k)\times (n-k)}(E).
	\]
	Since the top right $(m-k)\times k$ submatrix of $A$ is zero, $B(g_1,\ldots,g_{n-k})^\top$ is split surjective.  It follows that $\delta(Eg_1+\cdots+Eg_{n-k})\geqslant m-k$.
\end{proof}

The only purpose of our next observation is to assist us in proving Lemma~\ref{lemma:(f+edsg)z}, where we construct a split surjection in $\Hom_S(M,N)$ exhibiting a special form.

\begin{lemma}\label{lemma:special-section}
	Let $(e,f)$ be a split surjection in $\Hom_S(N\oplus M,N)$, and suppose that there is a split surjection $h$ in $\Hom_S(M,N)$.  Then there is a map $z$ in $\Hom_S(N,M)$ such that
	\[
	eE+fzE=E=Ehz.
	\]
\end{lemma}

\begin{proof}
	Let $p\in E$ and $q,r\in\Hom_S(N,M)$ be such that
	\[
	ep+fq=1=hr.
	\]
	Since
	\[
	hqfr+(1-hqfr)=1,
	\]
	we have
	\[
	hqE+(1-hqfr)E=E.
	\]
	Since $\sr(E)=1$, there is $d\in E$ such that
	\[
	\begin{array}{rcl}
	u&:=&hq+(1-hqfr)d\\
	\end{array}
	\]
	is a unit of $E$.  Let
	\[
	\begin{array}{rcl}
	y&:=&p(1-frd),\\
	z&:=&q(1-frd)+rd.\\
	\end{array}
	\]
	Then
	\[
	\begin{array}{rcl}
	ey+fz&=&e[p(1-frd)]+f[q(1-frd)+rd]\\
	&=&(ep+fq)(1-frd)+frd\\
	&=&(1-frd)+frd\\
	&=&1\\
	\end{array}
	\]
	and
	\[
	\begin{array}{rcl}
	hz&=&h[q(1-frd)+rd]\\
	&=&hq(1-frd)+hrd\\	
	&=&hq(1-frd)+d\\	
	&=&hq+(1-hqfr)d\\
	&=&u\\
	\end{array}
	\]
	are units of $E$.
\end{proof}

We conclude this section with a statement that allows us, during our proof of Lemma~\ref{lemma:rows}, to initiate a row reduction process stronger than the one that yields  Lemma~\ref{lemma:identity}.

\begin{lemma}[{cf.~Bass~\cite[``technical little argument" preceding Corollary~6.6]{Bas}}]\label{lemma:replacement}
	Choose $g_1,\ldots,g_n\in\Hom_S(M,N)$ for some integer $n\geqslant 2$, and suppose that $\delta(Eg_1+\cdots+Eg_n)\geqslant m$ for some $m\in\{1,\ldots,n-1\}$ so that, by the second statement of Lemma~\ref{lemma:identity}, there is a split surjection $y\in g_1+Eg_{m+1}+\cdots+Eg_n$.  Then, for every such $y$, there is a split surjection $(y_1,\ldots,y_m)^\top\in\Hom_S(M,N^{\oplus m})$ with $y_1=y$ and with $y_i\in g_i+Eg_{m+1}+\cdots+Eg_n$ for every $i\in\{2,\ldots,m\}$.
\end{lemma}

\begin{proof}
	By Lemma~\ref{lemma:identity}, there is
	\[
	A:=\begin{pmatrix}
	1  				& 0	& \cdots	& 0 		& *&\cdots 	&  *	\\
	0  			& \ddots	& \ddots 	& \vdots 	& \vdots 		&  			& \vdots	\\
	\vdots  		& \ddots	& \ddots	& 0		 	& \vdots 		& 		 	&  \vdots	\\
	0	 				& \cdots 	& 0 		& 1 		& * 	& \cdots	&  * 	\\
	\end{pmatrix}\in\mathbf{M}_{m\times n}(E)
	\]
	with $A(g_1,\ldots,g_n)^\top$ split surjective.  Let $(x_1,\ldots,x_m)^{\top}:=A(g_1,\ldots,g_n)^{\top}$ so that the map $x_i$ belongs to $g_i+Eg_{m+1}+\cdots+Eg_n$ for every $i\in\{1,\ldots,m\}$, and let $p$ and $q$ be sections of $x:=x_1$ and $y$, respectively.  Then
	\[
	xqyp+x(1_M-qy)p=1_N,
	\]
	and so
	\[
	xqE+x(1_M-qy)pE=E.
	\]
	Since $\sr(E)=1$, there is $d\in E$ such that
	\[
	t:=xq+x(1_M-qy)pd
	\]
	is a unit of~$E$.  Since $yq=1_N$, we have $y(1_M-qy)=0_{\Hom_S(M,N)}$ and $(1_M-qy)q=0_{\Hom_S(N,M)}$.  Hence
	\[
	\begin{tabular}{lcl}
	$u$&:=&$1_M+(1_M-qy)pdy$,\\
	$v$&:=&$(1_M-qy)+qt^{-1}y$,\\
	$w$&:=&$1_M-qx(1_M-qy)$\\
	\end{tabular}
	\]
	are units of $\End_S(M)$ with
	\[
	\begin{tabular}{lcl}
	$u^{-1}$&=&$1_M-(1_M-qy)pdy$,\\
	$v^{-1}$&=&$(1_M-qy)+qty$,\\
	$w^{-1}$&=&$1_M+qx(1_M-qy)$.\\
	\end{tabular}
	\]
	By direct computation, $xuvw=y$.  Furthermore, for every $i\in\{2,\ldots,m\}$, inspection reveals that
	\[
	x_i uvw\in x_i+Ex+Ey\subseteq g_i+Eg_1+Eg_{m+1}+\cdots+Eg_n.
	\]
	Hence there is
	\[
	B:=
	\begin{pmatrix}
	1		& 0		 			& \cdots 	& \cdots 	& 0 		& * 	& \cdots 	& *		\\
	*		& 1 				& 0	& \cdots	& 0 		& *	& \cdots	& *		\\
	\vdots	& 0  			& \ddots	& \ddots 	& \vdots 	& \vdots 		&  			& \vdots	\\
	\vdots	& \vdots  		& \ddots	& \ddots	& 0		 	& \vdots 		& 		 	& \vdots		\\
	* 		& 0	 				& \cdots 	& 0 		& 1 		& * 	& \cdots	& * \\
	\end{pmatrix}\in\mathbf{M}_{m \times n}(E)
	\]
	with $B(g_1,\ldots,g_n)^{\top}=(x_1,\ldots,x_m)^{\top}uvw$
	split surjective.  We now left-multiply by a matrix in $\mathbf{GL}_m(E)$ to clear every entry in the first column of $B$ except for the first entry.  This produces a matrix $C$ such that $(y_1,\ldots,y_m)^\top :=C(g_1,\ldots,g_n)^{\top}$
	is split surjective with $y_1=y$ and with $y_i\in g_i+Eg_{m+1}+\cdots+Eg_n$ for every $i\in\{2,\ldots,m\}$.
\end{proof}

In the next section, we continue working toward a proof of our main lemma by specializing to the case in which $E/\Jac(E)$ is a unit-regular ring.

%%%%%%%%%%%%%%%%%%%%%%%%%%%%%%%%%%%%%%%%%%%%%%%%%%%%%%%%
\section{The case of a module with an endomorphism ring that is unit-regular modulo its Jacobson radical}\label{sec:unit-regular}
%%%%%%%%%%%%%%%%%%%%%%%%%%%%%%%%%%%%%%%%%%%%%%%%%%%%%%%%

In the present section of the paper, $M$ and $N$ stand for right modules over a ring $S$ with $E:=\End_S(N)$ such that $E/\Jac(E)$ is a unit-regular ring.  We direct the reader to Section~\ref{sec:foundations}, Subsections ``Associative rings" and ``Stable rank", for background on unit-regular rings and their relatives.

Our primary goal here is to certify Lemma~\ref{lemma:rows}, the only statement from this section directly cited in the proof of our main lemma (Lemma~\ref{lemma:main}); the two other lemmas in this section (Lemmas~\ref{lemma:b+adsc} and~\ref{lemma:(f+edsg)z}) serve only to help us establish Lemma~\ref{lemma:rows}.  For the latter result, we do not need the full strength of our unit-regularity hypothesis on $E/\Jac(E)$ in this section; it would suffice to assume, in Lemma~\ref{lemma:b+adsc} for instance, that $E/\Jac(E)$ is an Artinian ring.  However, we find the more general version of Lemma~\ref{lemma:b+adsc} instructive in light of four related lemmas by other authors that exploit semisimplicity (Bass~\cite[Lemma~6.4]{Bas}, Eisenbud--Evans~\cite[Lemma~6]{EE}, Swan~\cite[Lemma~4]{Swa}, and Warfield~\cite[Lemma~3.1]{War}).  Besides that, we suspect that Lemma~\ref{lemma:b+adsc} could be used to remove a finiteness hypothesis in Condition (2) of our main theorem (Theorem~\ref{theorem:main-later}).  On the other hand, Example~\ref{example:Heinzer} below illustrates that, in Lemma~\ref{lemma:b+adsc}, the ring $E$ cannot be replaced by an arbitrary ring of stable rank 1.

\begin{lemma}\label{lemma:b+adsc}
	Let $a,b,c\in E$ be such that $aE+bE=E=Eb+Ec$.  Then there is $d\in E$ such that, for every central unit $s$ of $E$, the element $b+adsc$ is a unit of $E$.
\end{lemma}

\begin{proof}
	Every central unit of $E$ represents a central unit in  $E/\Jac(E)$, and every unit of $E/\Jac(E)$ can be represented by a unit of $E$.  Therefore, we may assume that $E$ is unit-regular.  Insofar as $E$ is unit-regular, there is a unit $v$ of $E$ with $b=bvb$.  Moreover, since $aE+bE=E=Eb+Ec$, there are $e,f,g,h\in E$ with $ae+bf=1=gb+hc$.  We now claim that we may take
	\[
	\begin{array}{rcl}
	d&:=&e(v^{-1}-b)h.
	\end{array}
	\]
	To prove this, let $s$ be a central unit of $E$.
	Since $b=bvb$, the elements $bv$ and $vb$ are idempotents with $(1-bv)b=0=b(1-vb)$.  Hence
	\[
	\begin{array}{rcl}
		t&:=&1-bf(1-bv),\\
		u&:=&bv+(1-bv)s,\\
		w&:=&1-(1-vb)gb
	\end{array}
	\]
	are units of $E$ with
		\[
	\begin{array}{rcl}
		t^{-1}&=&1+bf(1-bv),\\
		u^{-1}&=&bv+(1-bv)s^{-1},\\
		w^{-1}&=&1+(1-vb)gb.
	\end{array}
	\]
	As a result,
	\[
	\begin{array}{rcl}
	x&:=&(tu)(v^{-1}w)\\
	\end{array}
	\]
	is a unit of $E$.  We claim that $x=b+adsc$. We can verify this as follows.  First, since $ae+bf=1=gb+hc$, we have
	\[
	\begin{array}{rcl}
	t&=&1-bf(1-bv),\\
	&=&1-(1-ae)(1-bv),\\
	&=&bv+ae(1-bv)
	\end{array}
	\]
	and
	\[
	\begin{array}{rcl}
	w&=&1-(1-vb)gb,\\
	&=&1-(1-vb)(1-hc),\\
	&=&vb+(1-vb)hc.
	\end{array}
	\]
	Using our new expressions for $t$ and $w$, we get
	\[
	\begin{array}{rcl}
	tu&=&[bv+ae(1-bv)][bv+(1-bv)s]\\
	&=&bv+ae(1-bv)s
	\end{array}
	\]
	and
	\[
	\begin{array}{rcl}
	v^{-1}w&=&v^{-1}[vb+(1-vb)hc]\\
	&=&b+(v^{-1}-b)hc.
	\end{array}
	\]
	Hence
	\[
	\begin{array}{rcl}
	x&=&(tu)(v^{-1}w)\\
	&=&[bv+ae(1-bv)s][b+(v^{-1}-b)hc]\\
	&=&b+adsc
	\end{array}
	\]
	is a unit of $E$.
\end{proof}

As promised, we now demonstrate that the preceding lemma fails if we replace the ring $E$ with an arbitrary ring of stable rank 1.  The particular ring that we construct in the following example is a Euclidean domain with a countably infinite maximal spectrum.  This example is minimal in two ways:  First, a commutative ring of stable rank 1 is unit-regular if and only if it is \textit{$0$-dimensional reduced}, and our example shows that the previous lemma does not hold for an arbitrary \textit{$1$-dimensional domain} of stable rank 1.  Second, every commutative ring with a \textit{finite} maximal spectrum is unit-regular modulo its Jacobson radical, and our example indicates that a commutative ring of stable rank 1 need not have an \textit{uncountable} maximal spectrum in order to violate Lemma~\ref{lemma:b+adsc}.  Another notable property exhibited by our example is that, although our ring contains a field, neither the cardinality nor the characteristic of the field plays a role in our argument.
 
\begin{example}\label{example:Heinzer}
	\textit{There is a Euclidean domain $R$ of stable rank $1$ with a countably infinite maximal spectrum such that the following statement holds:  There are elements $a,b,c$ of $R$ with $aR+bR=R=Rb+Rc$ such that, for every $d\in R$, there is a (necessarily central) unit $s$ of $R$ with $b+adsc$ residing in a maximal ideal of $R$.}
\end{example}

\begin{proof} 
	Let $G$ denote a free $\ZZ$-module of countably infinite rank with standard basis elements $e_1,e_2,e_3,\ldots.$  Define a partial order $\preccurlyeq$ on $G$ by letting $\sum_{m=1}^\infty e_mg_m\preccurlyeq\sum_{m=1}^\infty e_mh_m$ if and only if $g_m\leqslant h_m$ for every positive integer $m$.  It is easily verified that every nonempty finite subset of $G$ has an infimum and a supremum with respect to $\preccurlyeq$.
	
	Let $P:=k\left[x_1^{\pm 1},x_2^{\pm 1},x_3^{\pm 1},\ldots\right]$ denote the ring of all Laurent polynomials in a countably infinite number of variables $x_1,x_2,x_3,\ldots$ over a field $k$.  Define
	\[
	\gamma\left(x_1^{g_1}\cdots x_m^{g_m}f\right):=e_1g_1+\cdots+e_mg_m\in G
	\]
	for every nonzero Laurent monomial $x_1^{g_1}\cdots x_m^{g_m}f\in P$, where $m$ is a positive integer; $g_1,\ldots,g_m$ are integers; and $f$ is a nonzero element of $k$. Next, define
	\[
	\gamma(p_1+\cdots+p_n):=\inf_{\preccurlyeq}\{\gamma(p_1),\ldots,\gamma(p_n)\}\in G
	\]
	for all positive integers $n$ and nonzero nonassociate Laurent monomials $p_1,\ldots,p_n\in P$.  Let $Q$ denote the field of fractions of $P$, and let $Q^*$ denote the group of units of $Q$.  Define
	\[
	\gamma\left(\frac{p}{q}\right):=\gamma(p)-\gamma(q)\in G
	\]
	for all nonzero Laurent polynomials $p,q\in P$.
	
	By Heinzer~\cite{Hein}, the map $\gamma:Q^*\rightarrow G$ thus defined is a surjective group homomorphism with kernel equal to the group of units of the ring $R:=\{r\in Q^*:0_G\preccurlyeq\gamma(r)\}\cup\{0_Q\}$, and $R$ is a one-dimensional B\'ezout domain of stable rank~1 with $\Max(R)=\{x_1R,x_2R,x_3R,\ldots\}$ and with $x_1R,x_2R,x_3R,\ldots$ distinct.  Since every prime ideal of $R$ is principal, $R$ is Noetherian by a well-known theorem of Cohen.  Since $R$ is Noetherian B\'ezout, $R$ is a principal ideal domain.  Since $R$ is a principal ideal domain of stable rank 1, the ring $R$ is a Euclidean domain by Estes--Ohm~\cite[Theorem~5.3]{EO}.
	
	Let $(a,b,c):=(x_1,x_2,1)$ so that $aR+bR=R=Rb+Rc$.  We must show that, for every $d\in R$, there is a unit $s$ of $R$ with $b+adsc$ residing in a maximal ideal of $R$.  If $d=0$, then we may let $s$ be an arbitrary unit of $R$ to finish.  Suppose that $d\neq 0$.  Let $\gamma(d):=e_1g_1+\cdots+e_mg_m$ for some positive integer $m$ and nonnegative integers $g_1,\ldots,g_m$.  Then $d=x_1^{g_1}\cdots x_m^{g_m}u$ for some unit $u$ of $R$.  Let $t:=x_1du^{-1}$, and note that
	\[
	\gamma(u^{-1})=\gamma\left(x_{m+1}-x_2\right)=\gamma\left(x_{m+1}+t\right)=0_G.
	\]
	Hence,
	\[
	s:=u^{-1}\left(\frac{x_{m+1}-x_2}{x_{m+1}+t}\right)
	\]
	is a unit of $R$, and
	\[
	v:=\frac{x_2+t}{x_{m+1}+t}
	\]
	is an element of $R$ with
	\[
	b+adsc=x_{m+1}v\in x_{m+1}R\in\Max(R). \qedhere
	\]
\end{proof}
 
The next lemma supplies a hypothesis guaranteeing the existence of a particular kind of split surjection in $\Hom_S(M,N)$.

\begin{lemma}\label{lemma:(f+edsg)z}
Let $(e,f)$ be a split surjection in $\Hom_S(N\oplus M,N)$. Let $g\in\Hom_S(M,N)$, and suppose that $Ef+Eg$ contains a split surjection.  Then there are maps $d\in E$ and $z\in\Hom_S(N,M)$ such that, for every central unit $s$ of $E$, the map $(f+edsg)z$ is a unit of~$E$ with $(f+edg)z=1$ in particular.
\end{lemma}

\begin{proof}
	Let $h$ be a split surjection in $Ef+Eg$.  By Lemma~\ref{lemma:special-section}, there is $z\in\Hom_S(N,M)$ such that
	\[
	eE+fzE=E=Ehz=Efz+Egz.
	\]
	Now, by Lemma~\ref{lemma:b+adsc}, there is $d\in E$ such that, for every central unit $s$ of $E$, the element $(f+edsg)z$	is a unit of $E$.  If $u:=(f+edg)z\neq 1$, then we may replace $z$ with $zu^{-1}$.
\end{proof}

The last result of this section affords a major inductive step in the proof of our main lemma.

\begin{lemma}\label{lemma:rows}
Select $g_1,\ldots,g_n\in\Hom_S(M,N)$ for some integer $n\geqslant 2$, and suppose that $\delta(Eg_1+\cdots+Eg_n)\geqslant m$ for some $m\in\{1,\ldots,n-1\}$.  Let $e\in E$ be such that the map $(e,g_1)\in\Hom_S(N\oplus M,N)$ is split surjective.  Then,
\begin{quote}
	there is $c_1\in E$ such that, for every central unit $s_1$ of $E$,
	
	\noindent\hspace{0.5in}there is $c_2\in E$ such that, for every central unit $s_2$ of $E$,
	
	\noindent\hspace{1in}$\ldots\ldots\ldots\ldots\ldots\ldots\ldots\ldots\ldots\ldots\ldots\ldots\ldots\ldots\ldots\ldots\ldots\ldots\textnormal{,}$
	
	\noindent\hspace{1.5in}there is $c_m\in E$ such 
	that, for every central unit $s_m$ of~$E$,	
	\[\delta(E(g_1+ec_1s_1g_n)+E(g_2+c_2s_2g_n)+\cdots+E(g_m+c_ms_mg_n)+Eg_{m+1}+\cdots+Eg_{n-1})\geqslant m.\]
\end{quote}
\end{lemma}

\begin{proof}
By the second statement of Lemma~\ref{lemma:identity}, there are $a_{1,m+1},\ldots,a_{1,n}\in E$ with
\[
g_1+\sum_{j=m+1}^n a_{1,j}g_j
\]
split surjective.  Since $(e,g_1)$ is also split surjective, Lemma~\ref{lemma:(f+edsg)z} tells us that there is $d_1\in E$ such that, for every central unit $s_1$ of $E$, the map
\[
g_1+ed_1s_1\sum_{j=m+1}^n a_{1,j}g_j
\]
is split surjective.  For every $j\in\{m+1,\ldots,n\}$, let $b_{1,j}:=d_1a_{1,j}$.  Note that $c_1:=b_{1,n}$ does not depend on $s_1$.

We now fix a central unit $s_1$ of $E$ and the concomitant split surjective map
\[
y:=g_1+\left(\sum_{j=m+1}^{n-1} eb_{1,j}s_1g_j\right)+ec_1s_1g_n.
\]
By Lemma~\ref{lemma:replacement}, there are $a_{i,j}\in E$ such that, if
\[
B_1:=
\begin{pmatrix}
1		& 0		 			& \cdots 	& \cdots 	& 0 		& eb_{1,m+1}s_1 	& \cdots 	& eb_{1,n-1}s_1&ec_1s_1		\\
0		& 1  				& 0	& \cdots	& 0 		& a_{2,m+1}	& \cdots	& a_{2,n-1}&a_{2,n}		\\
\vdots	& 0  			& \ddots	& \ddots 	& \vdots 	& \vdots 		&  			& \vdots&\vdots	\\
\vdots	& \vdots  		& \ddots	& \ddots	& 0		 	& \vdots 		& 		 	& \vdots&\vdots		\\
0 		& 0	 				& \cdots 	& 0 		& 1 		& a_{m,m+1} 	& \cdots	& a_{m,n-1}&a_{m,n}
\end{pmatrix}\in\mathbf{M}_{m \times n}(E),
\]
then $Y_1:=(y_{1,1},\ldots,y_{1,m})^\top:=B_1(g_1,\ldots,g_n)^\top$ is split surjective.  Notice that $y_{1,1}=y$.

Now let $k\in\{1,\ldots,m-1\}$, and suppose inductively that there are $b_{i,j},c_i\in E$ and central units $s_i$ of $E$ such that, if
\[
B_k:=
\setcounter{MaxMatrixCols}{20}
\begin{pmatrix}
	1&0&\cdots&\cdots&0&0&\cdots&\cdots&0&eb_{1,m+1}s_1&\cdots&eb_{1,n-1}s_1&ec_1s_1\\
	0&1&0&\cdots&0&0&\cdots&\cdots&0&b_{2,m+1}s_2&\cdots&b_{2,n-1}s_2&c_2s_2\\
	\vdots&0&\ddots&\ddots&\vdots&\vdots&&&\vdots&\vdots&&\vdots&\vdots\\
	\vdots&\vdots&\ddots&\ddots&0&\vdots&&&\vdots&\vdots&&\vdots&\vdots\\
	0&0&\cdots&0&1&0&\cdots&\cdots&0&b_{k,m+1}s_k&\cdots&b_{k,n-1}s_k&c_ks_k\\
	0&0&\cdots&\cdots&0&1&0&\cdots&0&a_{k+1,m+1}&\cdots&a_{k+1,n-1}&a_{k+1,n}\\
	\vdots&\vdots&&&\vdots&0&\ddots&\ddots&\vdots&\vdots&&\vdots&\vdots\\
	\vdots&\vdots&&&\vdots&\vdots&\ddots&\ddots&0&\vdots&&\vdots&\vdots\\
	0&0&\cdots&\cdots&0&0&\cdots&0&1&a_{m,m+1}&\cdots&a_{m,n-1}&a_{m,n}\\
\end{pmatrix},
\]
then $Y_k:=(y_{k,1},\ldots,y_{k,m})^\top:=B_k(g_1,\ldots,g_n)^\top$ is split surjective.  Let $Z_k:=(z_{k,1},\ldots,z_{k,m})$ be a section of $Y_k$.  Since
\[
\left(g_{k+1}+\sum_{j=m+1}^n a_{k+1,j}g_j\right)z_{k,k+1}=1,
\]
Lemma~\ref{lemma:b+adsc} tells us that there is $d_{k+1}\in E$ such that, for every central unit $s_{k+1}$ of $E$, the map
\[
\left(g_{k+1}+d_{k+1}s_{k+1}\sum_{j=m+1}^n a_{k+1,j}g_j\right)z_{k,k+1}
\]
is a unit of $E$.  For every $j\in\{m+1,\ldots,n\}$, let $b_{k+1,j}:=d_{k+1}a_{k+1,j}$.  Note that $c_{k+1}:=b_{k+1,n}$ may depend on $c_1,s_1,\ldots,c_k,s_k$ but does not depend on $s_{k+1}$.

Now fix a central unit $s_{k+1}$ of $E$, and define the unit
\[
u_{k+1}:=\left[g_{k+1}+\left(\sum_{j=m+1}^{n-1} b_{k+1,j}s_{k+1}g_j\right)+c_{k+1}s_{k+1}g_n\right]z_{k,k+1}
\]
of $E$.  Let
\[
B_{k+1}:=
\setcounter{MaxMatrixCols}{20}
\begin{pmatrix}
1&0&\cdots&\cdots&0&0&\cdots&\cdots&0&eb_{1,m+1}s_1&\cdots&eb_{1,n-1}s_1&ec_1s_1\\
0&1&0&\cdots&0&0&\cdots&\cdots&0&b_{2,m+1}s_2&\cdots&b_{2,n-1}s_2&c_2s_2\\
\vdots&0&\ddots&\ddots&\vdots&\vdots&&&\vdots&\vdots&&\vdots&\vdots\\
\vdots&\vdots&\ddots&\ddots&0&\vdots&&&\vdots&\vdots&&\vdots&\vdots\\
0&0&\cdots&0&1&0&\cdots&\cdots&0&b_{k+1,m+1}s_{k+1}&\cdots&b_{k+1,n-1}s_{k+1}&c_{k+1}s_{k+1}\\
0&0&\cdots&\cdots&0&1&0&\cdots&0&a_{k+2,m+1}&\cdots&a_{k+2,n-1}&a_{k+2,n}\\
\vdots&\vdots&&&\vdots&0&\ddots&\ddots&\vdots&\vdots&&\vdots&\vdots\\
\vdots&\vdots&&&\vdots&\vdots&\ddots&\ddots&0&\vdots&&\vdots&\vdots\\
0&0&\cdots&\cdots&0&0&\cdots&0&1&a_{m,m+1}&\cdots&a_{m,n-1}&a_{m,n}\\
\end{pmatrix};
\]
let  $Y_{k+1}:=(y_{k+1,1},\ldots,y_{k+1,m})^\top:=B_{k+1}(g_1,\ldots,g_n)^\top$; and let $Z_{k+1}:=(z_{k+1,1},\ldots,z_{k+1,m})$, where
\[
z_{k+1,k+1}:=z_{k,k+1}u_{k+1}^{-1}
\]
and
\[
z_{k+1,j}:=z_{k,j}-z_{k,k+1}u_{k+1}^{-1}y_{k+1,k+1}z_{k,j}
\]
for every $j\in\{1,\ldots,m\}\setminus\{k+1\}$.  By direct computation, $Y_{k+1}Z_{k+1}$ is the $m\times m$ identity matrix, so $Y_{k+1}$ is split surjective.  This completes our inductive step.

Induction now implies that there are $b_{i,j},c_i\in E$ and central units $s_i$ of $E$ such that, if
\[
B_m:=
\begin{pmatrix}
1		& 0		 			& \cdots 	& \cdots 	& 0 		& eb_{1,m+1}s_1 	& \cdots 	& eb_{1,n-1}s_1&ec_1s_1		\\
0		& 1  				& 0	& \cdots	& 0 		& b_{2,m+1}s_2	& \cdots	& b_{2,n-1}s_2&c_2s_2		\\
\vdots	& 0  			& \ddots	& \ddots 	& \vdots 	& \vdots 		&  			& \vdots&\vdots	\\
\vdots	& \vdots  		& \ddots	& \ddots	& 0		 	& \vdots 		& 		 	& \vdots&\vdots		\\
0 		& 0	 				& \cdots 	& 0 		& 1 		& b_{m,m+1}s_m 	& \cdots	& b_{m,n-1}s_m&c_ms_m \\
\end{pmatrix}\in\mathbf{M}_{m \times n}(E),
\]
then $B_m(g_1,\ldots,g_n)^\top$ is split surjective.  Moreover, our proof indicates that, for every index $i\in\{1,\ldots,m\}$, the element $c_i$ may depend on $c_1,s_1,\ldots,c_{i-1},s_{i-1}$ but does not depend on $s_i,c_{i+1},s_{i+1},\ldots,c_m,s_m$.

Now, let $C$ be the left $m\times (n-1)$ submatrix of $B_m$ so that
\[
C:=
\begin{pmatrix}
1		& 0		 			& \cdots 	& \cdots 	& 0 		& eb_{1,m+1}s_1 	& \cdots 	& eb_{1,n-1}s_1\\
0		& 1  				& 0	& \cdots	& 0 		& b_{2,m+1}s_2	& \cdots	& b_{2,n-1}s_2\\
\vdots	& 0  			& \ddots	& \ddots 	& \vdots 	& \vdots 		&  			& \vdots	\\
\vdots	& \vdots  		& \ddots	& \ddots	& 0		 	& \vdots 		& 		 	& \vdots		\\
0 		& 0	 				& \cdots 	& 0 		& 1 		& b_{m,m+1}s_m 	& \cdots	& b_{m,n-1}s_m\\
\end{pmatrix}\in\mathbf{M}_{m \times (n-1)}(E),
\]
and let
\[
(h_1,\ldots,h_{n-1})^{\top}:=(g_1+ec_1s_1g_n,g_2+c_2s_2g_n,\ldots,g_m+c_ms_mg_n,g_{m+1},\ldots,g_{n-1})^{\top}.
\]
Then $C(h_1,\ldots,h_{n-1})^{\top}=B_m(g_1,\ldots,g_n)^\top$ is split surjective.  The lemma now follows.
\end{proof}

We establish our main lemma in the next section after adapting three results from a previous paper by the author~\cite{Bai1} to suit our purposes here.

%%%%%%%%%%%%%%%%%%%%%%%%%%%%%%%%%%%%%%%%%%%%%%%%%%%%%%%%
\section{Our main lemma}\label{sec:main-lemma}
%%%%%%%%%%%%%%%%%%%%%%%%%%%%%%%%%%%%%%%%%%%%%%%%%%%%%%%%

In this section, we assume that $M$ and $N$ are right modules over a ring $S$ that is an algebra over a commutative ring~$R$, and we let $E:=\End_S(N)$ denote the $R$-algebra of all $S$-linear endomorphisms of $N$.  We also fix the following hypotheses, which manifest as Conditions (1) and (2) in our main theorem (Theorem~\ref{theorem:main-later}):
\begin{enumerate}
	\item $X:=j$-$\Spec(R)\cap\Supp_R(N)$ is a Noetherian subspace of the Zariski space $\Spec(R)$.
	\item $N$ is a finitely presented $S$-module, and $E$ is a module-finite $R$-algebra. 
\end{enumerate}

Proofs of the first three lemmas of this section are similar to proofs of three analogous lemmas from a previous paper by the author~\cite{Bai1}.  Accordingly, rather than provide proofs of the next three lemmas, we simply point out the results on which they depend and the lemmas from the earlier paper~\cite{Bai1} to which they correspond.  We remind the reader that the set $\mathit{\Lambda}$ appearing in Lemmas~\ref{lemma:canc-Lambda} and~\ref{lemma:sur-reduction} is defined in Section~\ref{sec:foundations}.  

Our main lemma rounds out this section.

\begin{lemma}\label{lemma:spl-closed}
	Let $F$ be a left $E$-submodule of $\Hom_S(M,N)$, and let $m$ be a nonnegative integer.  Then the set $Y_m:=\{\p\in X:\delta(F_{\p})\leqslant m\}$ is closed in $X$.
\end{lemma}

\begin{proof}
	This echoes the proof of~\cite[Lemma~8.8]{Bai1}.  The proof here does not rely on the requirement in~\cite{Bai1} that $S$ be a module-finite $R$-algebra.
\end{proof}

\begin{lemma}\label{lemma:canc-Lambda}
	Let $F$ be a finitely generated left $E$-submodule of $\Hom_S(M,N)$.  Then the set $\mathit{\Lambda}$ of test points of $\delta(F_-)$ in $X$ is finite.  On top of that, for every $\p\in X\setminus\mathit{\Lambda}$, there is $\q\in\mathit{\Lambda}$ with $\q\subsetneq \p$ and $\delta(F_{\q})=\delta(F_{\p})$.
\end{lemma}

\begin{proof}
	The reader may simply mimic the proof of~\cite[Lemma~5.3]{Bai1}.  Here, our proof rests on Lemmas~\ref{lemma:spl-infty} and \ref{lemma:spl-closed} from this paper.  As in Lemma~\ref{lemma:spl-closed}, the ring $S$ need not be a module-finite $R$-algebra here.
\end{proof}

\begin{lemma}\label{lemma:sur-reduction}
	Suppose that $\Hom_S(M,N^{\oplus n})$ contains an $X$-split map $(f_1,\ldots,f_n)^{\top}$ for some integer $n\geqslant 2$.  Let $F:=Ef_1+\cdots+Ef_n$, and let $\mathit{\Lambda}$ be the set of test points of $\delta(F_-)$ in $X$.   Let $U\in\mathbf{GL}_n(E)$, and let  $(p_1,\ldots,p_n)^\top:=U(f_1,\ldots,f_n)^\top$. Suppose that $(p_1,\ldots,p_{n-1})^\top$ is $\mathit{\Lambda}$-split.  Then $(p_1,\ldots,p_{n-1})^\top$ is $X$-split.
\end{lemma}

\begin{proof}
	The proof of~\cite[Lemma~5.7]{Bai1} can serve as a guide to the reader, although here we must appeal to Lemmas~\ref{lemma:canc-delta-one} and \ref{lemma:canc-Lambda} from this paper.  Once again, the ring $S$ need not be a module-finite $R$-algebra here.  On the other hand, the requirement in \cite[Lemma~5.7]{Bai1} that $X$ have finite dimension is harmless here since it is implied by the conditions of the present lemma; see Remark~\ref{remark:finite}.
\end{proof}

We are now ready to establish our main lemma.  This statement confirms and generalizes Conjecture~3.2.9 from the author's PhD dissertation~\cite{Bai3}.  Our main lemma also confirms \cite[Conjecture~3.2.8]{Bai3} and affirmatively answers \cite[Question~8.23]{Bai1} if $R=S=\End_S(N)$ and if, in our definition of a $\q$-split map from Section~\ref{sec:foundations}, we replace the number $1+\dim_X(\q)$ with $t+\dim_X(\q)$ for some fixed positive integer $t$.  The general versions of~\cite[Conjecture~3.2.8]{Bai3} and \cite[Question~8.23]{Bai1} are still open, as are \cite[Conjecture~3.2.7]{Bai3} and \cite[Question~7.2]{Bai1}.

\begin{lemma}[Main Lemma]\label{lemma:main}
	Suppose that there is a map $(f_1,\ldots,f_n)^{\top}\in\Hom_S(M,N^{\oplus n})$ that is $X$-split for some integer $n\geqslant 2$, and let $e\in E$ be such that $(e,f_1)\in\Hom_S(N\oplus M,N)$ is split surjective.  Then there are $d_1,\ldots,d_{n-1}\in E$ such that
	\[
	(f_1+ed_1f_n,f_2+d_2f_n,\ldots,f_{n-1}+d_{n-1}f_n)^{\top}
	\]
	is $X$-split.  Necessarily,
	\[
	\begin{pmatrix}
	e & f_1+ed_1f_n \\
	\end{pmatrix}
	=\begin{pmatrix}
	e & f_1 \\
	\end{pmatrix}
	\begin{pmatrix}
	1_N & d_1f_n \\
	0 & 1_M \\
	\end{pmatrix}\in\Hom_S(N\oplus M,N)
	\]
	will be split surjective.
\end{lemma}

\begin{proof}
	First of all, the conditions listed at the beginning of this section imply that $E_\p\cong(\End_S(N))_\p$ is a ring of stable rank 1 and that $E_\p/\Jac(E_\p)$ is a unit-regular ring for every $\p\in X$.  Hence, whenever we localize at a member $\p$ of $X$, we may apply the lemmas from Sections~\ref{sec:sr1} and~\ref{sec:unit-regular} to the right $S_\p$-modules $M_\p$ and $N_\p$ and the $R_\p$-algebra $E_\p$.  We will use this fact tacitly at various points in the proof at hand.
	
	Now, let $\mathit{\Lambda}$ denote the set of test points of $\delta(F_-)$ in $X$.  By Lemma~\ref{lemma:canc-Lambda}, the set $\mathit{\Lambda}$ is finite.  List the members of $\mathit{\Lambda}$ so that no member contains any of its predecessors.  Let $\q\in\mathit{\Lambda}$, and suppose inductively that there are $a_1,\ldots,a_{n-1}\in E$ such that
	\[
	(g_1,\ldots,g_{n-1})^{\top}:=(f_1+ea_1f_n,f_2+a_2f_n,\ldots,f_{n-1}+a_{n-1}f_n)^{\top}
	\]
	is $\p$-split for every predecessor $\p$ of $\q$.  Then, necessarily, $(e,g_1)$ will be split surjective.  Let $J$ be the intersection of the predecessors of $\q$, and let $g_n:=f_n$.  To complete our inductive step, it suffices to find $b_1,\ldots,b_{n-1}\in E$ and $r_1,\ldots,r_{n-1}\in J\setminus\q$ such that
	\[
	(h_1,\ldots,h_{n-1})^{\top}:=(g_1+eb_1r_1g_n,g_2+b_2r_2g_n,\ldots,g_{n-1}+b_{n-1}r_{n-1}g_n)^{\top}
	\]
	is $\q$-split:  Since $r_1,\ldots,r_{n-1}\in J\setminus\q$, Nakayama's Lemma will imply that $(h_1,\ldots,h_{n-1})^{\top}$ is also $\p$-split for every predecessor $\p$ of $\q$, and, necessarily, $(e,h_1)$ will be split surjective.
	
	Let $F:=Ef_1+\cdots+Ef_n$; let $G:=Eg_1+\cdots+Eg_{n-1}$; and let $m:=\delta(F_{\q})$.  By Lemma~\ref{lemma:spl-infty}, we know that $m\leqslant n$.
	
	Suppose first that $m=n$.  Then $m\geqslant 2$, and so $\delta(G_\q)\geqslant n-1$ by Lemma~\ref{lemma:canc-delta-one}.  Now, $(g_1,\ldots,g_{n-1})^\top$ is $\q$-split, which implies that we may take $b_1:=\cdots:=b_{n-1}:=0_E$ and that we may take $r_1,\ldots,r_{n-1}$ to be arbitrary members of $J\setminus\q$ to complete our inductive step.  
	
	For the remainder of the proof, then, suppose that $m\leqslant n-1$.  Since $(f_1,\ldots,f_n)^{\top}$ is $\q$-split, we must have $m\geqslant 1+\dim_X(\q)$.  (We use this fact near the end of the proof.)  By Lemma~\ref{lemma:rows},
	\begin{quote}
		there is $c_1\in E_\q$ such that, for every central unit $s_1$ of $E_\q$,
		
		\noindent\hspace{0.5in}there is $c_2\in E_\q$ such that, for every central unit $s_2$ of $E_\q$,
		
		\noindent\hspace{1in}$\ldots\ldots\ldots\ldots\ldots\ldots\ldots\ldots\ldots\ldots\ldots\ldots\ldots\ldots\ldots\ldots\ldots\ldots\textnormal{,}$
		
		\noindent\hspace{1.5in}there is $c_m\in E_\q$ such 
		that, for every central unit $s_m$~of~$E_\q$,	
	\begin{gather*}
	\delta\left(E_\q\left(\frac{g_1}{1}+\frac{e}{1}\cdot c_1s_1\cdot\frac{g_n}{1}\right)+E_\q\left(\frac{g_2}{1}+c_2s_2\cdot\frac{g_n}{1}\right)+\cdots+E_\q\left(\frac{g_m}{1}+c_ms_m\cdot\frac{g_n}{1}\right)\right.\hspace{0.99in}\\
	\hspace{3.725in}\left.+E_\q\left(\frac{g_{m+1}}{1}\right)+\cdots+E_\q\left(\frac{g_{n-1}}{1}\right)\right)\geqslant m.
	\end{gather*}
	\end{quote}
	Let $k\in\{0,\ldots,m-1\}$, and suppose inductively that we have chosen appropriate elements $c_1,\ldots,c_k$ of $E_\q$ and central units $s_1,\ldots,s_k$ of $E_\q$ relative to the last display.  Based on these choices, select $c_{k+1}\in E_\q$ appropriately; find $b_{k+1}\in E$ and $r_{k+1}\in J\setminus\q$ such that $c_{k+1}=b_{k+1}/r_{k+1}\in E_\q$; and let $s_{k+1}:=r^2_{k+1}/1\in E_\q$ so that $c_{k+1}s_{k+1}=b_{k+1}r_{k+1}/1\in E_\q$.  By induction, there are $b_1,\ldots,b_m\in E$ and $r_1,\ldots,r_m\in J\setminus \q$ such that, if
	\[
	(h_1,\ldots,h_{n-1})^{\top}:=(g_1+eb_1r_1g_n,g_2+b_2r_2g_n,\ldots,g_m+b_mr_mg_n,g_{m+1},\ldots,g_{n-1})^{\top}
	\]
	and $H:=Eh_1+\cdots+Eh_{n-1}$, then $\delta(H_\q)\geqslant m\geqslant 1+\dim_X(\q)$. Hence $(h_1,\ldots,h_{n-1})^\top$ is $\q$-split.  So we may complete our inductive step by setting $b_{m+1}:=\cdots:=b_{n-1}:=0_E$ and by letting $r_{m+1},\ldots,r_{n-1}$ be arbitrary members of $J\setminus\q$.
	
	Now, by induction on the members of $\mathit{\Lambda}$, there are $d_1,\ldots,d_{n-1}\in E$ such that
	\[
	(f_1+ed_1f_n,f_2+d_2f_n,\ldots,f_{n-1}+d_{n-1}f_n)^{\top}
	\]
	is $\mathit{\Lambda}$-split and, therefore, $X$-split by Lemma~\ref{lemma:sur-reduction}.  This certifies the first statement of the lemma, and the second statement of the lemma is evident, given the first.
\end{proof}

With the proof of our main lemma now complete, we are poised to establish our main theorem.  We accomplish this goal in the next section.  

%%%%%%%%%%%%%%%%%%%%%%%%%%%%%%%%%%%%%%%%%%%%%%%%%%%
\section{Our main theorem and two corollaries}\label{sec:main-theorem}
%%%%%%%%%%%%%%%%%%%%%%%%%%%%%%%%%%%%%%%%%%%%%%%%%%%

In this section, we prove our main theorem (Theorem~\ref{theorem:main-later}) and two corollaries (Corollaries~\ref{corollary:gen-Bass} and \ref{corollary:gen-DSPY}).  As promised in our introduction, Corollary~\ref{corollary:gen-Bass} contains the cancellation theorems of
Bass~\cite[Theorem~9.3]{Bas} and 
Dress~\cite[Theorem~2]{Dress}, and Corollary~\ref{corollary:gen-DSPY} recovers Theorem~\ref{theorem:short} from this paper as well as the
De~Stefani--Polstra--Yao Cancellation Theorem~\cite[Theorem~3.14]{DSPY}.

We rehash our main theorem here to aid the reader.  The statement below is identical to Theorem~\ref{theorem:main}  with the exception of Condition (3), which we rephrase here using the $\delta$~operator.

\begin{theorem}[Main Theorem, cf.~Theorem~\ref{theorem:main}]\label{theorem:main-later}
	Let $K$, $L$, $M$, and $N$ be right modules over a ring $S$ that is an algebra over a commutative ring $R$, and let $E:=\End_S(N)$ denote the $R$-algebra of all $S$-linear endomorphisms of $N$.  Assume the following:
	\begin{enumerate}
		\item $X:=j$-$\Spec(R)\cap\Supp_R(N)$ is a Noetherian subspace of the Zariski space $\Spec(R)$.
		\item $N$ is a finitely presented $S$-module, and $E$ is a module-finite $R$-algebra.
		\item There is a finitely generated left $E$-submodule $F$ of $\Hom_S(M,N)$ such that, for every $\p\in X$, we have $\delta(F_\p)\geqslant 1+\dim_X(\p)$.
	\end{enumerate}
	Then $N\oplus L\cong N\oplus M\Longrightarrow L\cong M$.  More generally, if $K$ is a direct summand of a direct sum of finitely many copies of $N$ over $S$, then $K\oplus L\cong K\oplus M\Longrightarrow L\cong M$.
\end{theorem}

\begin{proof}
	There is a right $S$-module $Q$ with $Q\oplus K\cong N^{\oplus m}$ for some positive integer $m$.  Hence $N^{\oplus m}\oplus L\cong N^{\oplus m}\oplus M$, and so, by induction on $m$, we may assume that $N\oplus L\cong N\oplus M$.  Now, there is $(e,f_1)\in\Hom_S(N\oplus M,N)$ such that
	\[
	\begin{pmatrix}
	e & f_1 \\
	* & * \\
	\end{pmatrix}
	\in\Hom_S(N\oplus M,N\oplus L)
	\]
	is an isomorphism.	By Condition (3), there are $f_2,\ldots,f_n\in F$ (for some integer $n\geqslant 2$) such that $F=Ef_2+\cdots+Ef_n$, and so $(f_1,\ldots,f_n)^{\top}\in\Hom_S(M,N^{\oplus n})$ is $X$-split.  Hence, by Conditions (1) and (2), we may apply our main lemma (Lemma~\ref{lemma:main}) iteratively $n-1$ times to obtain $f_0\in F$ such that $f_1+ef_0$ is $X$-split.  Lemma~\ref{lemma:DF}, combined with Conditions (1) and (2), then implies that $f_1+ef_0$ is split surjective over~$S$.  Let $g\in\Hom_S(N,M)$ be a section of $f_1+ef_0$, and let
	\[
	U:=
	\begin{pmatrix}
	1_N & f_0 \\
	0 & 1_M \\
	\end{pmatrix}
	\begin{pmatrix}
	1_N & 0 \\
	g-ge & 1_M \\
	\end{pmatrix}	
	\begin{pmatrix}
	1_N & -f_1-ef_0 \\
	0 & 1_M \\
	\end{pmatrix}\in\End_S(N\oplus M).
	\]
	Then $U$ is a unit of $\End_S(N\oplus M)$.  Hence
	\[
	\begin{pmatrix}
	e & f_1 \\
	* & * \\
	\end{pmatrix}
	U
	=\begin{pmatrix}
	1_N & 0 \\
	* & * \\
	\end{pmatrix}
	\in\Hom_S(N\oplus M,N\oplus L)
	\]
	is an isomorphism.  By the Five Lemma, $L\cong M$.
\end{proof}

\begin{remark}\label{remark:finite}
	The conditions of our main theorem collectively imply that $\dim(X)$ is finite:  If not, then there is $\q\in\Min(X)$ with $\delta(F_\q)\geqslant 1+\dim_X(\q)=\infty$, and so Lemma~\ref{lemma:spl-infty} tells us that $N_\q=0$, contrary to the hypothesis that $\q\in X\subseteq\Supp_R(N)$.  By the same reasoning, we do not need to assume in Corollary~\ref{corollary:gen-DSPY} below that $\dim(X)$ is finite since this property is forced by the constraints there.	However, we do assume that the dimension of $Y:=\Max(R)\cap\Supp_R(N)$ is finite in Lemma~\ref{lemma:spl-fg-L} and Corollary~\ref{corollary:gen-Bass} below; the reason is that, in each of these findings, we must account for the possibility that $\Hom_S(P,N)$ is not finitely generated as a left module over $\End_S(N)$.
\end{remark}

\begin{lemma}\label{lemma:spl-fg-L}
	Let $N$ and $P$ be right modules over a ring $S$ that is an algebra over a commutative ring $R$, and let $E:=\End_S(N)$ denote the $R$-algebra of all $S$-linear endomorphisms of~$N$.  Assume the following:
	\begin{enumerate}
		\item $Y:=\Max(R)\cap\Supp_R(N)$ is a finite-dimensional Noetherian subspace of the Zariski space $\Spec(R)$.
		\item $N$ is a finitely presented $S$-module.
		\item $P$ is a direct summand of a direct sum of finitely presented right $S$-modules, and $N_\m^{\oplus (1+\dim(Y))}$ is a direct summand of $P_\m$ over $S_\m$ for every $\m\in Y$.
	\end{enumerate}
	Then there is a finitely generated left $E$-submodule $F$ of $\Hom_S(P,N)$ such that, for every $\p$ in 
	 $X:=j$-$\Spec(R)\cap\Supp_R(N)$, we have  $\delta(F_{\p})\geqslant 1+\dim_X(\p)$.
\end{lemma}

\begin{proof}
	The proof is similar to that of \cite[Lemma~4.2]{Bai1} and does not rely on the requirement in \cite{Bai1} that $S$ be a module-finite $R$-algebra.  The proof here depends only on Lemma~\ref{lemma:spl-closed} from this paper.
\end{proof}

We can now state and prove our joint generalization of Bass~\cite[Theorem~9.3]{Bas} and Dress~\cite[Theorem~2]{Dress}. The following corollary of our main theorem recovers Bass when we assume that $P$ is a projective $S$-module and that $N=S=E$.  Corollary~\ref{corollary:gen-Bass} reduces to Dress when we require $M$ to be a direct summand of a direct sum of finitely presented $S$-modules and we require $P$ to be a direct summand of a direct sum of finitely many copies of $N$ over $S$.

\begin{corollary}[{Joint Generalization of Bass and Dress}]\label{corollary:gen-Bass}
	Let $K$, $L$, $M$, $N$, and $P$ be right modules over a ring $S$ that is an algebra over a commutative ring $R$, and let $E:=\End_S(N)$ denote the $R$-algebra of all $S$-linear endomorphisms of $N$.
	Assume the following:
	\begin{enumerate}
		\item $Y:=\Max(R)\cap\Supp_R(N)$ is a finite-dimensional Noetherian subspace of the Zariski space $\Spec(R)$.
		\item $N$ is a finitely presented $S$-module, and $E$ is a module-finite $R$-algebra.
		\item $P$ is a direct summand of a direct sum of finitely presented right $S$-modules; $P$ is a direct summand of $M$ over $S$; and $N_\m^{\oplus (1+\dim(Y))}$ is a direct summand of $P_\m$ over $S_\m$ for every $\m\in Y$.
	\end{enumerate}
	Then $N\oplus L\cong N\oplus M\Longrightarrow L\cong M$.  More generally, if $K$ is a direct summand of a direct sum of finitely many copies of~$N$ over $S$, then $K\oplus L\cong K\oplus M\Longrightarrow L\cong M$.
\end{corollary}	

\begin{proof}
	From Condition (1), we glean that $X:=j$-$\Spec(R)\cap\Supp_R(N)$ is a Noetherian space.  Condition (2) here coincides with Condition (2) of our main theorem.  Conditions (1)--(3), in tandem with Lemma~\ref{lemma:spl-fg-L}, indicate that $\Hom_S(P,N)$ contains a finitely generated left $E$-submodule $F$ of $\Hom_S(M,N)$ such that $\delta(F_\p)\geqslant 1+\dim_X(\p)$ for every $\p\in X$.  To finish our proof, we now simply appeal to our main theorem.
\end{proof}

We close this section with a joint generalization of Theorem~\ref{theorem:short} and the De~Stefani--Polstra--Yao Cancellation Theorem~\cite[Theorem~3.14]{DSPY}.  Theorem~\ref{theorem:short} constitutes the special case of the following result in which $K=N$ and $M=P$ and $R=S$.    De~Stefani--Polstra--Yao covers the special case of Corollary~\ref{corollary:gen-DSPY} in which $M=P$ and $N=R=S$.  

\begin{corollary}[{Joint Generalization of Theorem~\ref{theorem:short} and De~Stefani--Polstra--Yao}]\label{corollary:gen-DSPY}
	Let $K$,~$L$, $M$, $N$, and $P$ be right modules over a ring $S$ that is an algebra over a commutative ring~$R$.
	Assume the following:
	\begin{enumerate}
		\item $R$ is a Noetherian ring.
		\item $N$ is a finitely generated $S$-module, and $S$ is a module-finite $R$-algebra.
		\item $P$ is a finitely generated direct summand of $M$ over $S$, and $N_\p^{\oplus (1+\dim_X(\p))}$ is a direct summand of $P_\p$ over $S_\p$ for every $\p\in X:=j$-$\Spec(R)\cap\Supp_R(N)$.
	\end{enumerate}
	Then $N\oplus L\cong N\oplus M\Longrightarrow L\cong M$.  More generally, if $K$ is a direct summand of a direct sum of finitely many copies of~$N$ over $S$, then $K\oplus L\cong K\oplus M\Longrightarrow L\cong M$.
\end{corollary}	

\begin{proof}
	Condition (1) guarantees that $X$ is a Noetherian space.  Conditions (1) and (2) work together to ensure that $N$ is a finitely presented $S$-module and that $E:=\End_S(N)$ is a module-finite $R$-algebra.  Conditions (1)--(3) collectively imply that $F:=\Hom_S(P,N)$ is a finitely generated left $E$-submodule of $\Hom_S(M,N)$ such that $\delta(F_\p)\geqslant 1+\dim_X(\p)$ for every $\p\in X$.  An application of our main theorem completes the proof.
\end{proof}

Note that the preceding corollary covers Example~\ref{example:1}.  In the next section, we demonstrate that this example dodges many previously published cancellation theorems, including all those cited in our introduction. 

%%%%%%%%%%%%%%%%%%%%%%%%%%%%%%%%%%%%%%%%%%%%%%%%%%%
\section{A cancellation example}\label{sec:examples}
%%%%%%%%%%%%%%%%%%%%%%%%%%%%%%%%%%%%%%%%%%%%%%%%%%%

In our introductory section, we claim that Example~\ref{example:1} is not covered by any of the cancellation theorems preceding Theorem~\ref{theorem:short}.  We prove this assertion in the present section of the paper after establishing the two lemmas below.

\begin{lemma}\label{lemma:help}
	Assume the following:
		\begin{enumerate}
			\item $S$ is a Noetherian local factorial domain of dimension at least $2$ with maximal ideal~$\q$.
			\item $M:=\q^{\oplus g}\oplus S^{\oplus h}$ for some positive integers $g$ and $h$.
			\item $N$ is an $S$-module such that $N^{\oplus (g+h)}$ is a direct summand of $M$ over $S$.
		\end{enumerate}
	Then $N=0$.
\end{lemma}

\begin{proof}
	Suppose, by way of contradiction, that $N\neq 0$.  Then $N$ is a rank-1 torsion-free $S$-module and can, therefore, be identified with a nonzero ideal of $S$.  Moreover, letting $f:=g+h$, we find that
	\[
	N^{\oplus f}\cong \q^{\oplus g}\oplus S^{\oplus h}.
	\]
	By Heitmann--Wiegand~\cite[Theorem~8]{Heit-Wie}, the display implies that the ideals $N^f$ and $\q^g$ are isomorphic.  Since $\q^g$ has height at least $2$ in the Noetherian factorial domain $S$, the ideal $\q^g$ has grade at least $2$ on $S$, and so the natural map $S\rightarrow \Hom_S(\q,S)$ is an isomorphism.  Thus, there is $a\in S$ such that $N^f=a\q^g$.  Now, for every height-1 prime ideal $\p$ of $S$, the equation $N^f=a\q^g$ implies that $N_\p^f=aS_\p$.  This observation, combined with our assumption that $S$ is a factorial domain, ensures that there is an element $b$ of~$S$ (possibly a unit of $S$) with $b^fS=aS$.  Hence, $N^f=b^f\q^g$.  Letting $F:=b^{-1}N$, we may write $F^f=\q^g$, which implies that $F\subseteq \q$.  Thus, $F^f\subseteq \q^f \subseteq \q^g= F^f$, forcing $\q^f=\q^g$.  Since $f-g=h\geqslant 1$, Nakayama's Lemma then implies that $\q=0$, a contradiction.
\end{proof}

\begin{lemma}\label{lemma:more}
	Assume the following:
	\begin{enumerate}
		\item $R$ is a commutative ring with a Noetherian maximal spectrum of finite dimension $e$.
		\item $S$ is a $d$-dimensional affine $\CC$-algebra that is also an $R$-algebra. 
		\item $N$ is a faithful $S$-module such that $E:=\End_S(N)$ is a module-finite $R$-algebra.
	\end{enumerate}
	Then $1+d=\sr(E)\leqslant 1+e$.
\end{lemma}

\begin{proof}
	Since $S$ is a $d$-dimensional affine $\CC$-algebra, we have \[1+d=\sr(S)\] by Suslin~\cite[Corollary~8.4]{Sus2}.
	Since $N$ is a faithful $S$-module, the ring $S$ is a subring of $E$, and so \[\sr(S)\leqslant\sr(E)\] by Vaserstein~\cite[Theorem~1]{Vas2}.  Since the $R$-algebra $S$ is a central subring of the module-finite $R$-algebra $E$, we see that $E$ is a module-finite $S$-algebra.  Marrying this with the fact that the $d$-dimensional affine $\CC$-algebra $S$ has a $d$-dimensional Noetherian maximal spectrum, we get \[\sr(E)\leqslant 1+d\] by Bass's Stable Range Theorem~\cite[Theorem~11.1]{Bas}. Concatenating the last three displays, we get $1+d=\sr(E)$.  Finally, since $E$ is a module-finite algebra over a commutative ring $R$ with a Noetherian maximal spectrum of finite dimension $e$, we get $\sr(E)\leqslant 1+e$ by another application of Bass's Stable Range Theorem.
\end{proof}

For the reader's convenience, we reprint Example~\ref{example:1} here:

\begin{example}[cf.~Example~\ref{example:1}]\label{example:1-7}
	\textit{Assume the following:
		\begin{enumerate}
			\item $S$ is an affine $\CC$-domain of dimension $d\geqslant 3$.
			\item $K:=\q$ is a prime ideal of $S$ such that $S_\q$ is factorial and $2\leqslant \textnormal{height}(\q)\leqslant d-1$.
			\item $M:=\q^{\oplus d}\oplus S$.
		\end{enumerate}
		Then, for every $S$-module~$L$, we have $K\oplus L\cong K\oplus M\Longrightarrow L\cong M$.}
\end{example}

Some cancellation theorems are easy to rule out as potential precedents for this example.   For instance, in Goodearl--Warfield~\cite[Theorem~18]{GW},
Krull--Schmidt~\cite[Theorem~X.7.5]{Lang},
and Evans \cite[Theorem~2]{Eva}, it is the case that $\sr(\End_S(K))=1$, but here $K$ is a proper ideal of an affine $\CC$-algebra $S$ of dimension $d\geqslant 3$ with $\grade_S(K)\geqslant 2$, so $\sr(\End_S(K))=\sr(S)=1+d\geqslant 4$ by Suslin~\cite[Corollary~8.4]{Sus2}.   In Hs\"u~\cite[Theorem~1]{Hsu}, Wiegand~\cite[Theorem~1.2]{Wie1}, and many other theorems in  
\cite{Cha}
\cite{Gur-Lev}
\cite{Hassler0}
\cite{Hassler}
\cite{HW}
\cite{Karr}
\cite{Levy}
\cite{LW}
\cite{Rush}
\cite{Wie1}
\cite{Wie2}, the ring $S$ is a module-finite algebra over a commutative ring of dimension at most 2, but here $S$ is a commutative ring of dimension at least 3.  In Matlis~\cite[Theorems~2.4 and~2.5 and Propositions~2.7 and~3.1]{Mat}, the $S$-modules $K$, $L$, and $M$ are injective whereas, here, they are nonzero finitely generated modules over the positive-dimensional Noetherian domain $S$ and are, therefore, not injective~\cite[Theorem~3.1.17]{BH}.  In Bass~\cite[Theorem~9.3]{Bas}, De~Stefani--Polstra--Yao~\cite[Theorem~3.14]{DSPY}, and numerous results in
\cite{DK}
\cite{Duc}
\cite{Gupta}
\cite{Heit}
\cite{Kap}
\cite{MS}
\cite{Qui}
\cite{Rao}
\cite{Sta}
\cite{Sus}, it is assumed that $K$, $L$, or $M$ is a projective $S$-module, but here $K$ is an ideal of height at least 2 in the commutative Noetherian ring $S$, and $K$ is a direct summand of $L$ and $M$ over~$S$, so none of the three $S$-modules in question is projective.

Dismissing Vasconcelos~\cite[Corollary]{Vasc}, Warfield~\cite[Theorems~1.2 and~1.6]{War}, and Dress~\cite[Theorem~2]{Dress} requires more work.

Recall that, in Vasconcelos, $M\cong J^{\oplus n}$ for some ideal $J$ of $S$ and some nonnegative integer~$n$.  Suppose, by way of contradiction, that this is true in our example.  Then, by rank considerations, $n=1+d$.  Applying Lemma~\ref{lemma:help}, we find that $J_\q=0$.  As a result, $M_\q=0$, contrary to our hypothesis that the domain $S$ embeds into $M$. So Vasconcelos does not apply to our example.

Warfield ensures that cancellation holds if $\sr(\End_S(K))$ is finite and $K^{\oplus\sr(\End_S(K))}$ is a direct summand of $M$ over $S$.  Of course, by induction, we can replace Warfield's specifications with the requirement that $K$ be a direct summand of a direct sum of finitely many copies of some module $N$ over $S$ for which $\sr(\End_S(N))$ is finite and $N^{\oplus\sr(\End_S(N))}$ is a direct summand of $M$ over $S$.  Suppose, by way of contradiction, that this more general hypothesis holds in our example.  Since $K$ is a faithful direct summand of a direct sum of finitely many copies of $N$ over $S$, we see that $N$ is a faithful $S$-module.  Since $N$ is a submodule of the Noetherian $S$-module $M$ by hypothesis, we see that $N$ is a Noetherian $S$-module and that, consequently, $\End_S(N)$ is a module-finite $S$-algebra.  Combining the last two observations with our assumption that $S$ is a $d$-dimensional affine $\CC$-algebra, we can apply Lemma~\ref{lemma:more} with $R:=S$ to conclude that $\sr(\End_S(N))=1+d$.   Applying this to our hypothesis that $N^{\oplus \sr(\End_S(N))}$ is a direct summand of $M$ over $S$, we find that $N_\q^{\oplus (1+d)}$ is a direct summand of $M_\q$ over $S_\q$.  Lemma~\ref{lemma:help} then implies that $N_\q=0$, contrary to the claim that $N$ is a faithful Noetherian $S$-module.  This contradiction certifies that even the more general version of Warfield described above does not cover our example. 

In Dress, $S$ is an algebra over a commutative ring $R$ with a finite-dimensional Noetherian maximal spectrum; $N$ is an $S$-module such that $\End_S(N)$ is a module-finite $R$-algebra and such that $N_\m^{\oplus (1+\dim(\Max(R)))}$ is a direct summand of $M_\m$ over $S_\m$ for every $\m\in\Max(R)$; and $K$ is a direct summand of a direct sum of finitely many copies of $N$ over $S$.  Suppose, by way of contradiction, that these three conditions hold in our example.  As in our treatment of Warfield above, since $K$ is a faithful direct summand of a direct sum of finitely many copies of $N$ over $S$, we see that $N$ is a faithful $S$-module.  Since $S$ is a $d$-dimensional affine $\CC$-algebra by hypothesis, we may  set $e:=\dim(\Max(R))$ and apply Lemma~\ref{lemma:more} to conclude that 
\[
1+d\leqslant 1+e.
\]
Let $\sigma:R\rightarrow S$ be the structure map of the $R$-algebra $S$.  Since $\q$ is a prime ideal of~$S$, the set $\mathfrak{p}:=\sigma^{-1}(\q)$ is a prime ideal of $R$, and $\sigma(R\setminus\p)$ is a multiplicatively closed subset of~$S\setminus\q$.  By Dress's local condition, $N_\mathfrak{p}^{\oplus (1+e)}$ is a faithful direct summand of $M_\mathfrak{p}$ over~$S_\p$.  Combining two previous observations with the last display, we find that $N_\mathfrak{\q}^{\oplus (1+d)}$ is a nonzero direct summand of $M_\mathfrak{\q}$ over~$S_\q$, contrary to Lemma~\ref{lemma:help}.  Dress, therefore, cannot yield our example.

To close, we would like to return to the issue of why neither Bass nor De~Stefani--Polstra--Yao applies to our example.  Above, we simply point to the fact that $K$ is not a projective $S$-module, and indeed this shows that we cannot apply Bass or De~Stefani--Polstra--Yao \textit{directly}.  However, the reader might wonder whether there is a way to \textit{reduce} our example to one that is covered by Bass or De~Stefani--Polstra--Yao.  Below, we entertain three ruminations of this type and arrive at a dead end in each situation.

One approach would be to take the isomorphism $K\oplus L\cong K\oplus M$ and apply $\Hom_S(K,-)$ or $\Hom_S(-,S)$  to it.  In the first case, we would find that $S\oplus\Hom_S(K,L)\cong S\oplus S^{\oplus (1+d)}$, and then we would infer from Bass or De~Stefani--Polstra--Yao that $\Hom_S(K,L)\cong S^{\oplus (1+d)}$; however, with the last isomorphism, we would fail to recover $M=K^{\oplus d}\oplus S$.  In the second case, we would initially discover that $S\oplus \Hom_S(L,S)\cong S\oplus S^{\oplus (1+d)}$ and subsequently marshal Bass or De~Stefani--Polstra--Yao to conclude that $\Hom_S(L,S)\cong S^{\oplus (1+d)}$, but by then $M$ would have once again escaped us.

Another approach would be to apply $\Hom_S(-,K)$ to the isomorphism $K\oplus L\cong K\oplus M$.  In this approach, we would learn that $S\oplus \Hom_S(L,K)\cong S\oplus S^{\oplus d}\oplus K$ and then gather from Bass or De~Stefani--Polstra--Yao that $\Hom_S(L,K)\cong S^{\oplus d}\oplus K$.  Upon applying $\Hom_S(-,K)$ to the last isomorphism, we would ascertain further that $\Hom_S(\Hom_S(L,K),K)\cong M$.  However, at that point, we would need to know that $L\cong\Hom_S(\Hom_S(L,K),K)$ in order to prove that $L\cong M$.  In contrast, \textit{our main theorem reveals that} $L\cong\Hom_S(\Hom_S(L,K),K)$ \textit{as a corollary of the fact that} $L\cong M$.

A third approach would be to try to deduce, without using our main theorem, that the isomorphism $K\oplus L\cong K\oplus M$ implies the isomorphism $S\oplus L\cong S\oplus M$ in our example; this approach is inspired by Chase~\cite[Theorem~3.6]{Cha}.  The hope underlying this approach is that, upon procuring the second isomorphism, we would be able to apply Bass or De~Stefani--Polstra--Yao.  To dismiss this approach, it suffices to show that $M$ satisfies neither the local condition in Bass nor the local condition in De~Stefani--Polstra--Yao.  We can settle the case of Bass by mimicking our treatment of Dress above with $K$ and $N$ redefined as $S$ and with $M$ still equal to $\q^{\oplus d}\oplus S$ for some prime ideal $\q$ of $S$ such that $S_\q$ is factorial and $2\leqslant \textnormal{height}(\q)\leqslant d-1$.  For the other case, first recall that, in De~Stefani--Polstra--Yao, $S$~is a commutative ring such that $S_\q^{\oplus (1+\dim_X(\q))}$ is a direct summand of $M_\q$ over $S_\q$.  Suppose, by way of contradiction, that our example satisfies this condition.  Let $c:=\dim_X(\q)$.  Then $M_\q\cong G_\q\oplus S_\q^{\oplus (1+c)}$ for some $S$-module $G$.  Since $M_\q=\q_\q^{\oplus d}\oplus S_\q$ by hypothesis and since $\sr(S_\q)=1$, Evans's Cancellation Theorem~\cite[Theorem~2]{Eva} implies that $G_\q\oplus S_\q^{\oplus c}\cong \q_\q^{\oplus d}$.  Applying $\Hom_{S_\q}(-,\q_\q)$ to the last isomorphism, we find that
\[
\Hom_{S_\q}(G_\q,\q_\q)\oplus \q_\q^{\oplus c}\cong S_\q^{\oplus d}.
\]
Since $S_\q$ is a Noetherian local factorial domain of dimension at least 2 with maximal ideal~$\q_\q$, we have $\grade_{\q_\q}(\q_\q)=1$ and $\grade_{S_\q}(\q_\q)\geqslant 2$.  Hence, regarding the last display, the grade of $\q_\q$ on the left side is at most 1 whereas the grade of $\q_\q$ on the right side is at least 2, a contradiction.  So, even if there is a way to prove that $S\oplus L\cong S\oplus M$ without our main theorem, neither Bass nor De~Stefani--Polstra--Yao can be applied to this isomorphism to yield the conclusion that $L\cong M$. 

In summary, our main theorem reveals new information relative to more than eight cancellation results spanning four schools of module cancellation research.  Each school exploits one of the following types of mathematical objects:  module-finite algebras over commutative rings of dimension at most 2, direct sums of indecomposable modules, modules with endomorphism rings of finite stable rank, and projective modules.  Our main theorem responds to an adjuration of Eisenbud and Evans from 1973 calling for a unified cancellation theorem that, on one hand, circumvents projectivity conditions and, on the other hand, covers a robust collection of finitely generated modules over every commutative Noetherian ring.  Our main theorem fulfills this request by generalizing three established cancellation results while weakening a projectivity hypothesis in each one.  The cancellation example from this section provides one concrete way to distinguish our main theorem from its many predecessors.

%%%%%%%%%%%%%%%%%%%%%%%%%%%
\section*{Acknowledgements}
%%%%%%%%%%%%%%%%%%%%%%%%%%%

\noindent Versions of some of the results in this paper first appeared in a dissertation~\cite{Bai3} submitted to the Department of Mathematics and Statistics at Georgia State University in partial fulfillment of the requirements for the degree of Doctor of Philosophy.  This research did not receive any specific grant from funding agencies in the public, commercial, or not-for-profit sectors.

\bibliographystyle{amsplain}
\bibliography{baidya-references}

\end{document}